\documentclass[11pt,final]{article}

\usepackage{amssymb,amsmath,amsthm,mathrsfs,xfrac}
\usepackage{enumitem}
\usepackage{a4wide}
\usepackage{hyperref}
\usepackage{array}
\usepackage{xcolor}
\usepackage{colonequals}
\usepackage{arydshln}

\usepackage{mathabx}

\usepackage{tikz}

\numberwithin{equation}{section}
\allowdisplaybreaks[4]

\newtheorem{proposition}{Proposition}[section]
\newtheorem{theorem}[proposition]{Theorem}
\newtheorem{lemma}[proposition]{Lemma}
\newtheorem{definition}[proposition]{Definition}

\newtheorem{remark}[proposition]{Remark}

\renewenvironment{proof}{\smallskip\noindent{\textbf{Proof.}}%
  \hspace{1pt}}{\hspace{-5pt}{\nobreak\quad\nobreak\hfill\nobreak%
    $\square$\vspace{2pt}\par}\smallskip\goodbreak}

\newenvironment{proofof}[1]{\smallskip\noindent{\textbf{Proof~of~#1.}}%
  \hspace{1pt}}{\hspace{-5pt}{\nobreak\quad\nobreak\hfill\nobreak%
    $\square$\vspace{2pt}\par}\smallskip\goodbreak}

\newcommand{\Id}{\mathinner{\mathrm{Id}}}

\newcommand{\C}[1]{\mathbf{C}^{#1}}

\newcommand{\reali}{{\mathbb{R}}}

\renewcommand{\epsilon}{\varepsilon}
\renewcommand{\phi}{\varphi}
\renewcommand{\theta}{\vartheta}

\newcommand{\rnk}{\mathop{\rm rnk}}

\renewcommand{\d}[1]{\mathinner{\mathrm{d}{#1}}}

\usepackage{bbold}
\newcommand{\one}{\mathbb1}
\newcommand{\zero}{\mathbb0}
\newcommand{\claimend}{\hfill$\checkmark$\goodbreak}

\begin{document}

\title{Nash Equilibria in Traffic Networks \\ with Multiple
  Populations and Origins--Destinations
}

\author{Rinaldo M.~Colombo$^1$ \and Luca Giuzzi$^2$ \and Francesca
  Marcellini$^1$}

\date{}

\maketitle

\footnotetext[1]{INdAM Unit \& Department of Information Engineering,
  University of Brescia, Italy.\hfill\\ \texttt{rinaldo.colombo@unibs.it}, \texttt{francesca.marcellini@unibs.it}}

\footnotetext[2]{DICATAM, University of Brescia, Italy. \texttt{luca.giuzzi@unibs.it}}

\begin{abstract}
  \noindent Different populations of vehicles travel along a
  network. Each population has its origin, destination and travel
  costs --- which may well be unbounded. Under the only requirement of
  the continuity of the travel costs, we prove the existence of a Nash
  equilibrium for all populations. Conditions for its uniqueness are
  also provided. A few cases are treated in detail to show specific
  situations of interest.

  \medskip


  \medskip

\end{abstract}

\section{Introduction}
\label{sec:introduction}

A road network is used by different vehicles, grouped into
\emph{populations} according to their origin, their destination and
their travel costs. For each population, the origin is connected to
the destination by different \emph{routes}, each consisting of a
sequence of adjacent \emph{roads}. Traveling along a road bears a
\emph{cost}, in general different for each population, which may be
related to gas consumption, pollutant production or travel time. In
the latter case, it is reasonable to admit that it can become
infinite, for instance when congestions occur. A road can be shared by
drivers from various populations, with travel costs varying according
to the level of congestion. Each traveler chooses a route among those
available to its population and, correspondingly, pays a \emph{route
  cost} which is the sum of the costs to its population of the roads
constituting that route.

It is then natural to consider several \emph{concurrent} and
\emph{interdependent} \emph{games}, each played by the vehicles in a
single population. According to the \emph{selfish
  paradigm}~\cite{Roughgarden2004332, Roughgarden2005,
  Roughgarden200093}, each individual chooses a route such that no
different choice could be more convenient. In other words, all
populations distribute among the different routes so that each game is
at a \emph{Nash equilibrium} and the whole network is at a
\emph{global} Nash equilibrium.

Below, under rather general assumptions, we prove the existence of
such global Nash equilibria. At these configurations, for each
population, all vehicles in that population pay the same
cost. Moreover, at the equilibria we exhibit, no individual, whatever
the population it belongs to, finds convenient to change its route
choice. The costs of different populations may well be unrelated, in
the sense that a population can be interested in minimizing travel
time, another in gas consumption, and so on.

In the case of a single population, frameworks essentially equivalent
to the one presented below are classical, see for
instance~\cite{MR247935, MR801491, Nagurney2007355, Roughgarden2005,
  Roughgarden200093}. In general, existence proofs of Nash equilibria
rely on the considered game being a \emph{potential game},
see~\cite{beckmann, MR247935}, a property that here fails due to the
presence of multiple populations. A different approach to multiple
populations competing on the same network is presented
in~\cite{MR3337989}, where a measure based framework is employed. Nash
equilibria in multi-agent systems are considered, for instance,
in~\cite{MR3950676}.

As the number of populations and the complexity increase, it is
realistic to expect that multiple (global) Nash equilibria may
occur. Thus, conditions ensuring uniqueness in general situations
become more intricate and require stricter assumptions. First, we
present an equality satisfied by all (global) Nash equilibria. As long
as the number of populations is small and the networks are simple,
this condition may ensure uniqueness, as examples show. In the
Appendix, we also present a general uniqueness theorem.

We stress that in the present, multi-population, setting, a definition
of a \emph{globally optimal\/\footnote{Also referred to as
    \textit{social optimum}, conforming to Wardrop's \emph{Second
      principle}, see~\cite[p.~345]{Wardrop1952}.}}  strategy appears
as quite arbitrary, due to the presence of multiple cost functions. It
is then open to further investigations to extend to the present case
the many relevant results, see e.g.~\cite{Roughgarden2005,
  Roughgarden200093}, about the comparison between globally optimal
strategies and Nash equilibria.

The present multi-population setting allows the appearance of the well
known Braess paradox~\cite{BraessParadox} in a variety of new
situations. Since we comprehend also the case of populations using the
same network but consisting of different vehicles, say trucks and
cars, we show that Braess paradox may appear for both populations, see
\S~\ref{subsec:trucks-cars-same}. Otherwise, the insertion of a new
road for only one population may well increase travel cost at Nash
equilibrium for another population, not directly affected by the new
road, see \S~\ref{subsec:braess-parad-induc}.

To reduce formal complexity, we consider the case of $2$ populations,
i.e., of $2$ origins and $2$ destinations. The $2$ populations are
referred to as the \emph{hat} and \emph{check} population. The
extension to finitely many populations requires merely notational
modifications.

The next Section~\ref{sec:OK} presents the formal setting and the
general existence result. All analytic proofs are left to
Section~\ref{sec:technical-details}. Specific Cases are presented in
Section~\ref{sec:examples}. The final Appendix is devoted to a general
uniqueness result.

\section{Formal Framework and Existence Result}
\label{sec:OK}

We consider $N$ one way roads $r_1, \ldots, r_N$ having at least one
end point (a junction) in common with other roads. At each junction
there is at least one entering and one exiting road. Two roads are
adjacent if the end point of one of the two roads is the initial point
of the other one. A route $\gamma$ is an $m$--tuple of adjacent
(pairwise distinct) roads, say
$\gamma \equiv (r_{h_1}, \ldots, r_{h_m})$, such that the end point of
$r_{h_\ell}$ is the initial point of $r_{h_{\ell+1}}$, for
$\ell \in \{1, \ldots, m-1\}$. By \emph{network} $\mathcal{N}$ we mean
the set of routes.

We assume there are $2$ origin--destination pairs, say
$(\widehat{\mathcal{O}}, \widehat{\mathcal{D}})$ and
$(\widecheck{\mathcal{O}}, \widecheck{\mathcal{D}})$. Correspondingly,
we call $\widehat{\mathcal{N}}$ the sub--network of $\mathcal{N}$
consisting of routes
$\widehat{\gamma}_1, \ldots, \widehat{\gamma}_{\widehat{n}}$
connecting $\widehat{\mathcal{O}}$ to $\widehat{\mathcal{D}}$ and,
similarly we call $\widecheck{\mathcal{N}}$ the sub-network of
$\mathcal{N}$ consisting of routes
$\widecheck{\gamma}_1, \ldots, \widecheck{\gamma}_{\widecheck{n}}$
connecting $\widecheck{\mathcal{O}}$ to
$\widecheck{\mathcal{D}}$. Clearly, $\widehat{n} \leq N$ and
$\widecheck{n} \leq N$.

In the terminology of graph theory, both networks
$\widehat{\mathcal{N}}$ and $\widecheck{\mathcal{N}}$ are (connected)
directed acyclic graphs (DAG) with a single sink and a single source.

As it is usual, see e.g.~\cite{Holden2016}, construct the
$N \times \widehat{n}$ matrix $\widehat{\Gamma}$ and the
$N \times \widecheck{n}$ matrix $\widecheck{\Gamma}$ setting
\begin{eqnarray*}
  \widehat{\Gamma}_{hi}
  & \colonequals
  & \begin{cases}
    1
    & \text{road $r_h$ belongs to route $\widehat{\gamma}_i$},
    \\
    0
    & \text{otherwise.}
  \end{cases}
      \qquad \mbox{for }
      h \in \{1, \ldots, N\} \mbox{ and }i \in \{1, \ldots, \widehat{n}\} \,;
  \\
  \widecheck{\Gamma}_{hi}
  & \colonequals
  & \begin{cases}
    1 & \text{road $r_h$ belongs to route $\widecheck{\gamma}_i$},
    \\
    0 & \text{otherwise.}
  \end{cases}
        \qquad \mbox{for }
        h \in \{1, \ldots, N\} \mbox{ and }i \in \{1, \ldots, \widecheck{n}\}\,.
\end{eqnarray*}

A single road may well belong to more than one route and to one or
both the subnetworks $\widehat{\mathcal{N}}$ and
$\widecheck{\mathcal{N}}$. A structural assumption of use below on the
subnetworks $\widehat{\mathcal{N}}$ and $\widecheck{\mathcal{N}}$ is
the following condition:
\begin{enumerate}[label=\bf{($\mathbf\Gamma$)}]
\item \label{it:IR} Each route in $\widehat{\mathcal{N}}$ contains a
  road that is not contained in any other route in
  $\widehat{\mathcal{N}}$. Similarly, each route in
  $\widecheck{\mathcal{N}}$ contains a road that is not contained in
  any other route in $\widecheck{\mathcal{N}}$.
\end{enumerate}
\noindent Clearly, the above condition implies that both matrices
$\widehat{\Gamma}$ and $\widecheck{\Gamma}$ contain a copy of the
identity $\Id_{\widehat{n}}$ or $\Id_{\widecheck{n}}$. Hence
$\widehat{\Gamma}$ and $\widecheck{\Gamma}$ are both full rank:
$\rnk \widehat{\Gamma} = \widehat{n}$ and
$\rnk \widecheck{\Gamma} = \widecheck{n}$.

Up to a renormalization, see~\cite[Section~1]{MR2020123}, we assume
that the total amount of vehicles driving along each of the networks
$\widehat{\mathcal{N}}$ and $\widecheck{\mathcal{N}}$ is $1$.
According to their choices, $\widehat{\theta}_i$ travelers of the
first population choose the route $\widehat{\gamma}_i$ for
$i \in \{1, \ldots, \widehat{n}\}$ and $\widecheck{\theta}_i$
travelers of the second population choose $\widecheck{\gamma}_i$ for
$i \in \{1, \ldots, \widecheck{n}\}$. Clearly,
$\widehat\theta \equiv (\widehat\theta_1, \ldots, \widehat\theta_n)$,
respectively
$\widecheck\theta \equiv (\widecheck\theta_1, \ldots,
\widecheck\theta_n)$, is a point in the $\widehat{n}-1$ dimensional
simplex $S^{\widehat{n}}$, respectively in the $\widecheck{n}-1$
dimensional simplex $S^{\widecheck{n}}$. We identify $\widehat\theta$
and $\widecheck\theta$ with column vectors. Thus, the total number of
vehicles along the road $r_h$ is
$\widehat{\Gamma}_h \, \widehat{\theta} + \widecheck{\Gamma}_h \,
\widecheck{\theta} = \sum_{i=1}^{\widehat{n}} \widehat{\Gamma}_{hi} \,
\widehat{\theta}_i + \sum_{i=1}^{\widecheck{n}}
\widecheck{\Gamma}_{hi} \, \widecheck{\theta}_i$.

In general, each road $r_h$ is equipped with its \emph{costs}
$\widehat{\tau}_h = \widehat{\tau}_h (\widehat\eta, \widecheck\eta)$
for the hat population and
$\widecheck{\tau}_h = \widecheck{\tau}_h (\widehat\eta,
\widecheck\eta)$ for the check population, where $\widehat\eta$,
respectively $\widecheck\eta$, is the amount of hat, respectively
check, population traveling along $r_h$. When dealing with, say,
trucks and cars they may be the costs (travel times) for each of the
two populations, for instance. In the case of populations consisting
of homogeneous vehicles differing only in their origins--destinations,
it is reasonable for example to set
$\widehat{\tau}_h = \widecheck{\tau}_h$ for all
$h \in \{1, \ldots, N\}$. Whenever the vehicles of the two populations
are homogeneous and the total numbers of vehicles of each population
are the same, it is likely that $\widehat{\tau}_h$ and
$\widecheck{\tau}_h$ are functions of the sum
$\widehat{\eta} + \widecheck{\eta}$. However, the normalization of
different total numbers of travelers to $1$ suggests to consider also
general functions
$\widehat\tau_h = \widehat\tau_h(\widehat\eta,\widecheck\eta)$ and
$\widecheck\tau_h = \widecheck\tau_h(\widehat\eta,\widecheck\eta)$.

It is of interest to allow also for the case of fully congested
roads. Therefore, we allow these costs to attain the value
$+\infty$. To this aim, we introduce the following classes of
functions.

We call a map $\tau \colon [0,1]^2 \to \reali_+ \cup \{+\infty\}$
\emph{weakly increasing in both variables} if
\begin{displaymath}
  \begin{array}{c}
    \forall\,\widehat\eta, \widecheck\eta',\widecheck\eta'' \in [0,1] \qquad
    \widecheck\eta' \leq \widecheck\eta'' \implies \tau (\widehat\eta,\widecheck\eta') \leq \tau (\widehat\eta,\widecheck\eta'') \,;
    \\
    \forall\,\widehat\eta', \widehat\eta'',\widecheck\eta \in [0,1] \qquad
    \widehat\eta' \leq \widehat\eta'' \implies \tau (\widehat\eta',\widecheck\eta) \leq \tau (\widehat\eta'',\widecheck\eta)\,.
  \end{array}
\end{displaymath}
Then, we introduce the class of continuous functions weakly increasing
in both variables attaining values in $\reali_+ \cup \{+\infty\}$:
\begin{displaymath}
  \mathcal{C}
  \coloneq
  \left\{\tau \in \C0 ([0,1]^2; \reali_+ \cup \{+\infty\})
    \colon \tau \mbox{ is weakly increasing in both variables}\right\} \,.
\end{displaymath}
Below, most results require one of the following assumptions on all
cost functions:
\begin{enumerate}[label={\bf(C\arabic*)}]
\item\label{item:E1} For all $h \in \{1, \ldots, N\}$, both travel
  times $\widehat\tau_h, \widecheck\tau_h$ are in $\mathcal{C}$ and
  admit a continuous derivative at every point $\eta$ where they
  attain finite values.
\item \label{item:E2} For all $h \in \{1, \ldots, N\}$, both travel
  times $\widehat\tau_h, \widecheck\tau_h$ are in $\mathcal{C}$ and
  are convex.
\end{enumerate}

\noindent It is then useful to introduce the maps,
\begin{equation}
  \label{eq:10}
  \begin{array}{@{}c@{}}
    \begin{array}[b]{@{}c@{\;}c@{\;}c@{\;}c@{\;}c@{}}
      \widehat\tau
      & \colon
      & [0,1]^N \times [0,1]^N
      & \to
      & \reali_+^N
      \\
      &
      & \widehat\eta, \widecheck\eta
      & \mapsto
      & \widehat\tau (\widehat\eta,\widecheck\eta)
    \end{array}
        \mbox{ where } \left(\widehat\tau (\widehat\eta,\widecheck\eta)\right)_h = \widehat\tau_h (\widehat\eta_h,\widecheck\eta_h)
        \quad \mbox{ for } \quad h \in \{1, \ldots, N\} \,;
    \\[12pt]
    \begin{array}[b]{@{}c@{\;}c@{\;}c@{\;}c@{\;}c@{}}
      \widecheck\tau
      & \colon
      & [0,1]^N \times [0,1]^N
      & \to
      & \reali_+^N
      \\
      &
      & \widehat\eta, \widecheck\eta
      & \mapsto
      & \widecheck\tau (\widehat\eta,\widecheck\eta)
    \end{array}
        \mbox{ where } \left(\widecheck\tau (\widehat\eta,\widecheck\eta)\right)_h = \widecheck\tau_h (\widehat\eta_h,\widecheck\eta_h)
        \quad \mbox{ for } \quad h \in \{1, \ldots, N\} \,.
  \end{array}
\end{equation}

\noindent To shorten the notation, we often set
$\tau \equiv (\widehat\tau,\widecheck\tau)$,
$\eta \equiv (\widehat\eta,\widecheck\eta)$ and
$\theta \equiv (\widehat\theta,\widecheck\theta)$.

\medskip

Denote by $\widehat{T}_i (\theta)$ and $\widecheck{T}_i (\theta)$ the
route travel times along the routes $\widehat\gamma_i$ and
$\widecheck\gamma_i$, while $\widehat{T} (\theta)$ and
$\widecheck{T} (\theta)$ are the vectors of all travel times of each
population, so that
\begin{equation}
  \label{eq:11}
  \begin{array}{@{}cccc@{}}
    \widehat{T} (\theta)
    \colonequals
    \widehat\Gamma^\intercal \;
    \widehat \tau
    (\widehat\Gamma \, \widehat\theta, \widecheck\Gamma \, \widecheck\theta)
    & \mbox{ and } \qquad
    & \widehat{T}_i (\theta)
      =
      \sum_{h=1}^N \widehat\Gamma_{hi} \;
      \widehat \tau_h
      (\widehat\Gamma_h \, \widehat\theta, \widecheck\Gamma_h \, \widecheck\theta)\,,
    & i \in \{1, \ldots, \widehat n\}\,;
    \\
    \widecheck{T} (\theta)
    \colonequals
    \widecheck\Gamma^\intercal \;
    \widecheck \tau
    (\widehat\Gamma \, \widehat\theta, \widecheck\Gamma \, \widecheck\theta)
    & \mbox{ and } \qquad
    & \widecheck{T}_i (\theta)
      =
      \sum_{h=1}^N \widecheck\Gamma_{hi} \;
      \widecheck \tau_h
      (\widehat\Gamma_h \, \widehat\theta, \widecheck\Gamma_h \, \widecheck\theta)\,,
    & i \in \{1, \ldots, \widecheck n\}\,.
  \end{array}
\end{equation}
From a global point of view, it is natural to evaluate the quality of
a network through the mean route travel times\footnote{Also referred
  to as \textit{average latency} of the system or \textit{social cost}
  of the network.} resulting from the partitions $\widehat\theta$ and
$\widecheck\theta$:
\begin{equation}
  \label{eq:6}
  \begin{array}{@{}c@{}}
    {\widehat{T}_M} (\theta)
    =
    \widehat\theta^\intercal \; \widehat{T} (\theta)
    \quad \mbox{ or, equivalently, } \quad
    {\widehat{T}_M} (\theta)
    \colonequals
    \sum_{i=1}^n \widehat\theta_i \; \widehat{T}_i (\theta) \,;
    \\
    {\widecheck{T}_M} (\theta)
    =
    \widecheck\theta^\intercal \; \widecheck{T} (\theta)
    \quad \mbox{ or, equivalently, } \quad
    {\widecheck{T}_M} (\theta)
    \colonequals
    \sum_{i=1}^n \widecheck\theta_i \; \widecheck{T}_i (\theta) \,;
  \end{array}
\end{equation}

Recall the following basic definition inspired
from~\cite[Definition~3.1]{Holden2016}.

\begin{definition}
  \label{def:Eq}
  Given a state $\theta \in S^n \times S^n$, we call \emph{hat
    relevant}, respectively \emph{check relevant} those route travel
  times $\widehat{T}_i (\theta)$, respectively
  $\widecheck{T}_i (\theta)$, such that $\widehat\theta_i \neq 0$,
  respectively $\widecheck\theta_i \neq 0$. A state
  $\theta^* \in S^n \times S^n$ is an \emph{equilibrium state} if all
  relevant route travel times of each population coincide, i.e.,
  \begin{displaymath}
    \begin{array}{@{}c@{}}
      \mbox{for all } i,j \in \{1, \ldots, \widehat n\} \quad
      \mbox{ if } \widehat\theta^*_i \neq 0 \mbox{ and } \widehat\theta^*_j \neq 0,
      \mbox{ then } \widehat{T}_i (\theta^*) = \widehat{T}_j (\theta^*)\,;
      \\
      \mbox{for all } i,j \in \{1, \ldots, \widecheck n\} \quad
      \mbox{ if } \widecheck\theta^*_i \neq 0 \mbox{ and } \widecheck\theta^*_j \neq 0,
      \mbox{ then } \widecheck{T}_i (\theta^*) = \widecheck{T}_j (\theta^*)\,.
    \end{array}
  \end{displaymath}
  The common value of the relevant hat/check route travel times is the
  \emph{hat/check equilibrium time}.
\end{definition}
In other words, at equilibrium all drivers of the same population need
the same time to go from that population's origin to that population's
destination. Clearly, if $\widehat\theta^*$ is an extreme point of
$S^{\widehat n}$ and $\widecheck\theta^*$ is an extreme point of
$S^{\widecheck n}$, then $\theta^*$ is an equilibrium.


\noindent Then, clearly, at equilibrium the common value of the
relevant travel times is the mean route travel time. However, at
equilibrium, the mean route travel time needs not be optimal.

The idea of Nash equilibrium~\cite{Nash1951}, see
also~\cite[Definition~1.3]{BressanHan2013}, corresponds to a situation
where no player finds convenient to change strategy. If an equilibrium
$\theta^*$ lies along the boundary of
$S^{\widehat n} \times S^{\widecheck n}$, so that
$\widehat\theta_j^* = 0$ (or $\widecheck\theta^*_i=0$) simple
continuity considerations ensure that if
$\widehat{T}_j (\theta^*) < \widehat{T}_M (\theta^*)$ (or
$\widecheck{T}_j (\theta^*) < \widecheck{T}_M (\theta^*)$) then
passing to the $j$-th route is convenient for some hat (or check)
drivers.

\begin{definition}
  \label{def:candiate}
  An equilibrium state $\theta^*$ is a \emph{Nash equilibrium} if for
  $j \in \{1, \ldots, n\}$
  \begin{equation}
    \label{eq:9}
    \begin{array}{crcl}
      \forall\, i \in \{1, \ldots, \widehat n\}
      & \widehat\theta^*_i = 0
      & \implies
      & \widehat{T}_i (\theta^*) \geq {\widehat{T}_M} (\theta^*) \,;
      \\
      \forall\, i \in \{1, \ldots, \widecheck n\}
      & \widecheck\theta^*_i = 0
      & \implies
      & \widecheck{T}_i (\theta^*) \geq {\widecheck{T}_M} (\theta^*) \,.
    \end{array}
  \end{equation}
\end{definition}

In the search for equilibria, comparing travel times corresponding to
different distributions of drivers plays a key role. Assume that
$\epsilon$ drivers pass from the $i$-th route $\widehat \gamma_i$ to
the $j$-th route $\widehat \gamma_j$, so that $\widehat \theta$
becomes
$\widehat \theta-\epsilon \, e_i + \epsilon \, e_j$\footnote{As usual,
  $e_1, \ldots, e_n$ is the canonical basis in $\reali^n$, and we use
  the same notation in $\reali^{\widehat n}$ and
  $\reali^{\widecheck n}$.}. Then, the present framework is compatible
with the obvious observation that the travel time
$\widehat T_i (\theta)$ along the $i$-th route $\widehat \gamma_i$
does not increase.

\begin{lemma}
  \label{lem:mono}
  Let~\ref{item:E1} hold. For all $i,j \in \{1, \ldots, \widehat n\}$
  and for all $\theta \in S^{\widecheck n} \times S^{\widehat n}$,
  \begin{align}
    \label{eq:21}
    \widehat\theta_i
    >
    0
    \mbox{ and }
    \widehat{T}_i (\theta) < +\infty
    & \implies
      \forall \, \epsilon
      \mbox{ sufficiently small}
    & \widehat{T}_i (\theta)
    & \geq \widehat{T}_i (\widehat\theta - \epsilon\, e_i + \epsilon \, e_j, \widecheck\theta) \,;
    \\
    \label{eq:27}
    \widehat\theta_j
    <
    1
    \mbox{ and }
    \widehat{T}_j (\theta) < +\infty
    & \implies
      \forall \, \epsilon
      \mbox{ sufficiently small}
    & \widehat{T}_j (\theta)
    & \leq \widehat{T}_j (\widehat\theta + \epsilon\, e_j - \epsilon \, e_i,\widecheck\theta)
    \\
    \label{eq:39}
    &
    &  \mbox{ and }
    &
      \widehat{T}_j (\widehat\theta + \epsilon\, e_j - \epsilon \, e_i,\widecheck\theta) < +\infty
  \end{align}
  An entirely analogous statement holds for the check population.
\end{lemma}

\noindent A further general monotonicity property is provided by
Lemma~\ref{lem:monotone}.

In the literature, see e.g.~\cite[Definition~2.2.1]{Roughgarden2005},
Nash equilibria are also related to a small percentage of players
changing strategy. Differently from~\eqref{eq:21}--\eqref{eq:27}, this
condition requires a comparison between travel times along different
routes, as defined below. The following definition is inspired, for
instance, by~\cite[Definition~3.3]{Holden2016}
or~\cite[Definition~2.1]{Roughgarden200093}.

\begin{definition}
  \label{def:Nash}
  An equilibrium state
  $\theta^* \in S^{\widehat n} \times S^{\widecheck n}$ is an
  \emph{$\epsilon$--Nash Equilibrium} if for all sufficiently small
  $\epsilon$ and
  \begin{displaymath}
    \begin{array}{@{}c@{}}
      \mbox{if }
      i,j \in \{1, \ldots, \widehat n\}
      \mbox{ and }
      \widehat\theta^* - \epsilon \, e_i + \epsilon \, e_j \in S^{\widehat n}
      \quad \mbox{ then } \quad
      \widehat{T}_j (\widehat\theta^* - \epsilon \, e_i + \epsilon \, e_j,\widecheck\theta^*)
      \geq
      \widehat{T}_i (\theta^*) \,;
      \\
      \mbox{if }
      i,j \in \{1, \ldots, \widecheck n\}
      \mbox{ and }
      \widecheck\theta^* - \epsilon \, e_i + \epsilon \, e_j \in S^{\widecheck n}
      \quad \mbox{ then } \quad
      \widecheck{T}_j (\widehat\theta^*, \widecheck\theta^* - \epsilon \, e_i + \epsilon \, e_j)
      \geq
      \widecheck{T}_i (\theta^*) \,.
    \end{array}
  \end{displaymath}
\end{definition}

\noindent In other words, for $\epsilon$ drivers there is no gain in
changing from route $\widehat\gamma_i$ to route
$\widehat\gamma_j$. Note that in the literature, the term
\emph{``$\epsilon$--Nash Equilibrium''} often refers to approximate Nash equilibria, where no player would lower its cost by more than $\epsilon$ when changing strategy, see~\cite[\S~2.6.6]{NisaRougTardVazi07}.

\goodbreak

The following theorem ensures the equivalence of
Definition~\ref{def:candiate} and Definition~\ref{def:Nash} in all
cases of interest.

\begin{theorem}
  \label{thm:convex}
  If travel times are continuous, $\epsilon$--Nash equilibria are also
  Nash equilibria. Moreover, either of the conditions~\ref{item:E1}
  or~\ref{item:E2} ensures that any Nash equilibrium is also an
  $\epsilon$--Nash equilibrium.
\end{theorem}

\noindent The extension to $\widehat\tau_h \in \C1$ and
$\widecheck\tau_h$ convex for all $h$, or \emph{vice versa}, is
immediate.

Intuitive connections between globally optimal configurations and Nash
equilibria are confirmed and formalized by the next Proposition.

\begin{proposition}
  \label{prop:GlobNash}
  Let~\ref{item:E1} hold.
  \begin{enumerate}[label={\bf(NE\arabic*)}]
  \item\label{it:N1} If
    $\theta^* \in S^{\widehat n} \times S^{\widecheck n}$ is an
    equilibrium such that
    \begin{equation}
      \label{eq:29}
      \widehat{T}_M (\theta^*)
      =
      \min_{\widehat\theta\in S^{\widehat n}} \widehat{T}_M (\widehat\theta, \widecheck\theta^*)
      \quad \mbox{ and } \quad
      \widecheck{T}_M (\theta^*)
      =
      \min_{\widecheck\theta\in S^{\widecheck n}} \widecheck{T}_M (\widehat\theta^*, \widecheck\theta)
    \end{equation}
    then $\theta^*$ is a Nash equilibrium.
  \item\label{it:N2} If both $\widehat\theta^*$ and
    $\widecheck\theta^*$ are extreme points of $S^{\widehat n}$ and
    $S^{\widecheck n}$ satisfying~\eqref{eq:29}, then $\theta^*$ is a
    Nash equilibrium.
  \end{enumerate}
\end{proposition}

\noindent The following result, inspired
by~\cite[Theorem~1]{Nash1951}, is based on Brouwer Fixed Point
Theorem, see e.g.~\cite[Theorem~1.6.2]{MR0488102}. Continuity of the
road travel times suffices to ensure the existence of Nash equilibria.

\begin{theorem}
  \label{thm:existence}
  Let all road travel times
  $(\widehat\tau_1,\widecheck\tau_1), \ldots,
  (\widehat\tau_N,\widecheck\tau_N)$ be in
  $\C0 ([0,1]^2; (\reali_+ \cup\{+\infty\})^2)$. Then, there exists a
  Nash equilibrium
  $\theta^* \in S^{\widehat n}\times S^{\widecheck n}$, in the sense
  of Definition~\ref{def:candiate}.
\end{theorem}

We now proceed to present a condition satisfied by multiple concurrent
Nash equilibria. As shown in \S~\ref{sec:examples}, it is a useful
tool to prove the uniqueness of Nash equilibria.

\begin{proposition}
  \label{prop:unique0}
  Let~\ref{item:E1} hold.  If $\theta'$ and $\theta''$ are Nash
  equilibria, then
  \begin{equation}
    \label{eq:1}
    (\widehat\theta'' - \widehat\theta')^\intercal \, \left(\widehat{T}
      (\theta'') - \widehat{T} (\theta')\right)
    =
    0
    \quad \mbox{ and } \quad
    (\widecheck\theta'' - \widecheck\theta')^\intercal \,
    \left(\widecheck{T} (\theta'') - \widecheck{T} (\theta')\right)
    =
    0 \,.
  \end{equation}
\end{proposition}

Note that whenever $\widehat n = \widecheck n = 2$, the above
condition~\eqref{eq:1}, complemented with the constraints
$\widehat\theta'_1 + \widehat\theta'_2 = \widehat\theta''_1 +
\widehat\theta''_2$ and
$\widecheck\theta'_1 + \widecheck\theta'_2 = \widecheck\theta''_1 +
\widecheck\theta''_2$ leads to a (possibly non-linear) system of $4$
equations in $4$ variables and, hence, may ensure the uniqueness of
the Nash equilibrium.

\section{Specific Cases}
\label{sec:examples}

Here we highlight the meaning of the analytic structure introduced in
Section~\ref{sec:OK}.

As a first example we present a pathological situation where, in the
case of a single population, the equilibria in
Definition~\ref{def:candiate} and Definition~\ref{def:Nash} are
different. Let $N=2$, $n=2$,
$\Gamma = \left[{{1\;0}\atop{0\;1}}\right]$ and set
$\tau_1 (\eta) = 1+3\eta$, $\tau_2 (\eta) = 3-\eta$. Then, $(1,0)$ and
$(0,1)$ are equilibria but not Nash equilibria; $(1/2,1/2)$ is a Nash
equilibrium but not an $\epsilon$--Nash equilibrium. In particular,
there is no $\epsilon$--Nash equilibrium.

Below, to simplify the presentation, we define only those travel times
that are necessary, i.e., if the road $r_h$ does not belong to
$\widehat{\mathcal{N}}$, then $\widehat\tau_h$ is not defined.

\subsection{Nash Equilibria -- a Simple Case}
\label{subsec:twopopulations}

A well known feature of traffic on networks~\cite{JAMOUS2018296,
  REINOLSMANN2022454, ZANG2022103866} are the \emph{non-local}
consequences of changes in a road's travel time. We see below that the
insurgence of a delay on a road may have effects on routes
geographically distant and, apparently, completely independent.

Assume that two different populations of vehicles move in the
following network, where $N=5$, $\widehat{n} = 2$,
$\widecheck{n} = 2$:
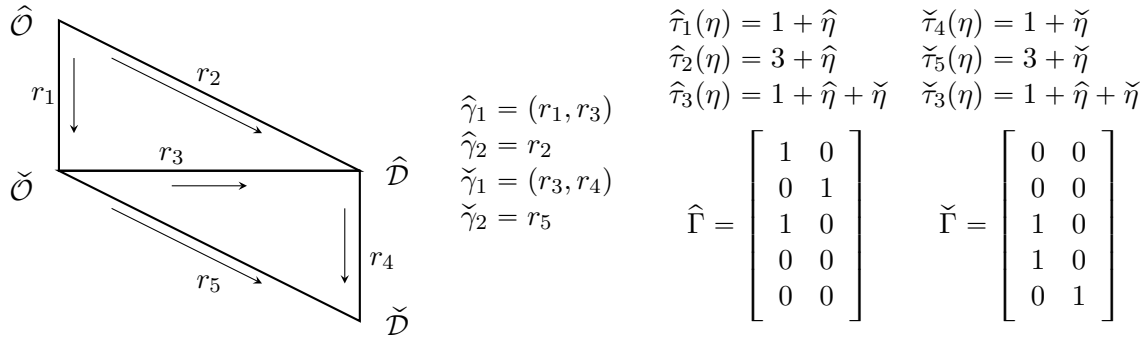
\begin{displaymath}
  \begin{tikzpicture}[baseline]
    \draw[thick] (-2,0) -- (2,0) -- (-2,2) -- cycle; %
    \draw[thick] (-2,0) -- (2,0) -- (2,-2) -- cycle; %
    \draw [-stealth] (-1.8,1.5) -- (-1.8, .5); %
    \draw [-stealth] (1.8,-0.5) -- (1.8,-1.5); %
    \draw [-stealth] (-0.5,-0.2) -- (0.5,-0.2); %
    \draw [-stealth] (-1.3,1.5) -- (0.7,0.5); %
    \draw [-stealth] (-1.3,-0.5) -- (0.7,-1.5); %
    \draw (-2.5,2) node {$\widehat{\mathcal{O}}$}; %
    \draw (2.5,0) node {$\widehat{\mathcal{D}}$}; %
    \draw (2.5,-2) node {$\widecheck{\mathcal{D}}$}; %
    \draw (-2.5,-0.2) node
    {$\widecheck{\mathcal{O}}$}; \draw (-2.2,1) node {$r_1$}; %
    \draw (0,1.25) node {$r_2$}; %
    \draw (-0.5,0.2) node {$r_3$}; %
    \draw (0,-1.5) node {$r_5$}; %
    \draw (0.2,-0.4);
    \draw (2.3,-1.2) node {$r_4$};
  \end{tikzpicture}
  \quad
  \begin{array}{l}
    \widehat{\gamma}_1 = (r_1 , r_3)
    \\
    \widehat{\gamma}_2 = r_2
    \\
    \widecheck{\gamma}_1 =  (r_3 , r_4)
    \\
    \widecheck{\gamma}_2 =  r_5
  \end{array}
  \quad
  \begin{array}{@{}c@{\;}c@{}}
    \begin{array}{l}
      \widehat{\tau}_1 (\eta) = 1 + \widehat{\eta}
      \\
      \widehat{\tau}_2 (\eta)= 3 + \widehat{\eta}
      \\
      \widehat{\tau}_3 (\eta)= 1 + \widehat{\eta} +\widecheck{\eta}
    \end{array}
    & \begin{array}{l}
        \widecheck{\tau}_4 (\eta)= 1 + \widecheck{\eta}
        \\
        \widecheck{\tau}_5 (\eta)= 3 + \widecheck{\eta}
        \\
        \widecheck{\tau}_3 (\eta)= 1 + \widehat{\eta} + \widecheck{\eta}
      \end{array}
    \\[20pt]
    \widehat{\Gamma}
    =
    \left[
    \begin{array}{cc}
      1
      & 0
      \\
      0
      & 1
      \\
      1
      & 0
      \\
      0
      & 0
      \\
      0
      & 0
    \end{array}
        \right]
      & \widecheck{\Gamma}
        =
        \left[
        \begin{array}{cc}
          0
          & 0
          \\
          0
          & 0
          \\
          1
          & 0
          \\
          1
          & 0
          \\
          0
          & 1
        \end{array}
            \right]
  \end{array}
\end{displaymath}
The first population moves from the origin $\widehat{\mathcal{O}}$ to
the destination $\widehat{\mathcal{D}}$; $\widehat\theta_1$ travelers
go along the central route $\widehat\gamma_{1}=(r_{1}, r_{3})$, while
$\widehat\theta_2$ travelers move along the northern highway
$\widehat\gamma_{2} = r_{2}$.  The second population moves from the
origin $\widecheck{\mathcal{O}}$ to the destination
$\widecheck{\mathcal{D}}$ and $\widecheck\theta_1$ travelers move from
the origin $\widecheck{\mathcal{O}}$ to the destination
$\widecheck{\mathcal{D}}$ along the central route
$\widecheck\gamma_{1} = (r_{3}, r_{4})$, while $\widecheck\theta_2$
travelers move along the southern highway
$\widecheck\gamma_{2} = r_{5}$. Road $r_3$ is used by both
populations. We thus have
\begin{displaymath}
  \begin{array}{r@{\,}c@{\,}l@{\qquad\qquad}r@{\,}c@{\,}l}
    \widehat T_{1} (\theta)
    & =
    & 2 +  2\widehat \theta_1 + \widecheck \theta_1 \,;
    & \widecheck T_1 (\theta)
    & =
    & 2 + \widehat \theta_1 + 2\widecheck \theta_1 \,;
    \\
    \widehat T_2(\theta)
    & =
    & 3 + \widehat \theta_2 \,;
    & \widecheck T_2 (\theta)
    & =
    & 3 + \widecheck \theta_2 \,.
  \end{array}
\end{displaymath}
The point $\theta^* \equiv \left((1/2,1/2),(1/2,1/2)\right)$ is a Nash
equilibrium with corresponding travel times
$\widehat{T}_i (\theta^*) = \widehat{T}_j (\theta^*) = 7/2$ for
$i,j \in \{1,2\}$. A straightforward application of
Proposition~\ref{prop:unique0} ensures that $\theta^*$ is the unique
Nash equilibrium.

Assume now that travel time increases on the northern highway
$\widehat\gamma_{2} = r_{2}$ by a fixed delay $\Delta$ due, for
instance, to construction works or to an accident, so that we have
\begin{displaymath}
  \begin{array}{l}
    \widehat{\tau}_1 (\eta) = 1 + \widehat{\eta}
    \\
    \widehat{\tau}_2 (\eta)= 3 + \widehat{\eta} +\Delta
    \\
    \widehat{\tau}_3 (\eta)= 1 + \widehat{\eta} +\widecheck{\eta}
  \end{array}
  \quad
  \begin{array}{l}
    \widecheck{\tau}_4 (\eta)= 1 + \widecheck{\eta}
    \\
    \widecheck{\tau}_5 (\eta)= 3 + \widecheck{\eta}
    \\
    \widecheck{\tau}_3 (\eta)= 1 + \widehat{\eta} + \widecheck{\eta}
  \end{array}
  \quad
  \begin{array}{r@{\,}c@{\,}l@{\qquad}r@{\,}c@{\,}l}
    \widehat T_{1} (\theta)
    & =
    & 2 +  2\widehat \theta_1 + \widecheck \theta_1
    & \widecheck T_1 (\theta)
    & =
    & 2 + \widehat \theta_1 + 2\widecheck \theta_1
    \\
    \widehat T_2(\theta)
    & =
    & 3 + \widehat \theta_2 + \Delta
    & \widecheck T_2 (\theta)
    & =
    & 3 + \widecheck \theta_2\,.
  \end{array}
\end{displaymath}
Straightforward computations show that
$\theta_\Delta \equiv \left(1/2 + 3\Delta/8, 1/2 - 3\Delta/8), (1/2 -
  \Delta/8, 1/2 + \Delta/8)\right)$ is a Nash equilibrium. The same
computations as above, on the basis of Proposition~\ref{prop:unique0},
ensure that $\Theta_\Delta$ is the unique Nash equilibrium. The travel
times at $\theta_\Delta$ are
$\widehat{T}_1 (\theta_\Delta) = \widehat{T}_2 (\theta_\Delta) = 7/2 +
5\Delta/8$ and
$\widecheck{T}_1 (\theta_\Delta) = \widecheck{T}_2 (\theta_\Delta) =
7/2 + \Delta/8$: they are worse than before for all routes, also for
those not directly influenced by the delay.

\subsection{Nash Equilibria with Unbounded Travel Times}
\label{subsec:nash-equilibria-with}

The non-local effects presented in \S~\ref{subsec:twopopulations} are
here more relevant due to the congestion caused on a road by a delay
appearing on an apparently independent road.

Two populations of vehicles move in the following network, where
$N=7$, $\widehat{n} = 2$, $\widecheck{n} = 2$:
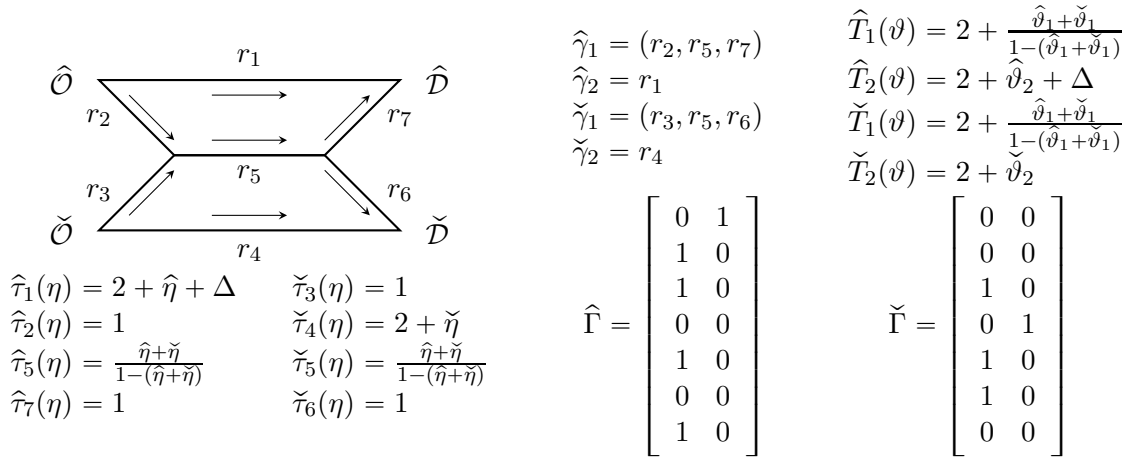
\begin{displaymath}
  \begin{array}{c}
    \begin{tikzpicture}[baseline]
      \draw[thick] (-2,1) -- (2,1) -- (1,0) -- (-1,0) -- cycle;%
      \draw[thick] (-2,-1) -- (2,-1) -- (1,0) -- (-1,0) -- cycle;%
      \draw (-2.5,1) node {$\widehat{\mathcal{O}}$}; %
      \draw (2.5,1) node {$\widehat{\mathcal{D}}$}; %
      \draw (-2.5,-1) node {$\widecheck{\mathcal{O}}$}; %
      \draw (2.5,-1) node {$\widecheck{\mathcal{D}}$}; %
      \draw [-stealth] (-0.5,0.8) -- (0.5,0.8); %
      \draw [-stealth] (-0.5,0.2) -- (0.5,0.2); %
      \draw [-stealth] (-0.5,-0.8) -- (0.5,-0.8); %
      \draw [-stealth] (-1.6,-0.8) -- (-1,-0.2); %
      \draw [-stealth] (-1.6,0.8) -- (-1,0.2); %
      \draw [-stealth] (1,0.2) -- (1.6,0.8); %
      \draw [-stealth] (1,-0.2) -- (1.6,-0.8); %
      \draw (0,1.3) node {$r_1$}; %
      \draw (-2,0.5) node {$r_2$}; %
      \draw (-2,-0.5) node {$r_3$}; %
      \draw (0,-1.3) node {$r_4$}; %
      \draw (0,-0.25) node {$r_5$}; %
      \draw (2,-0.5) node {$r_6$}; %
      \draw (2,0.5) node {$r_7$}; %
    \end{tikzpicture}
    \\
    \begin{array}{l}
      \widehat{\tau}_1 (\eta) = 2 + \widehat\eta + \Delta
      \\
      \widehat{\tau}_2 (\eta)= 1
      \\
      \widehat{\tau}_5 (\eta) = \frac{\widehat{\eta} + \widecheck{\eta}}{1- (\widehat{\eta} +\widecheck{\eta})}
      \\
      \widehat{\tau}_7 (\eta)= 1
    \end{array}
    \quad
    \begin{array}{l}
      \widecheck{\tau}_3 (\eta)= 1
      \\
      \widecheck{\tau}_4 (\eta)= 2 + \widecheck\eta
      \\
      \widecheck{\tau}_5 (\eta)= \frac{\widehat{\eta} + \widecheck{\eta}}{1- (\widehat{\eta} +\widecheck{\eta})}
      \\
      \widecheck{\tau}_6 (\eta) = 1
    \end{array}
  \end{array}
  \qquad
  \begin{array}{@{}c@{\qquad}c@{}}
    \begin{array}{@{}l}
      \widehat{\gamma}_1 = (r_2 , r_5, r_7)
      \\
      \widehat{\gamma}_2 = r_1
      \\
      \widecheck{\gamma}_1 =  (r_3 , r_5, r_6)
      \\
      \widecheck{\gamma}_2 =  r_4
    \end{array}
    & \begin{array}{l@{}}
        \widehat T_1 (\theta) = 2 + \frac{\widehat{\theta}_1 + \widecheck{\theta}_1}{1- (\widehat{\theta}_1 +\widecheck{\theta}_1)}
        \\
        \widehat T_2 (\theta) = 2 + \widehat\theta_2 + \Delta
        \\
        \widecheck T_1 (\theta) = 2 + \frac{\widehat{\theta}_1 + \widecheck{\theta}_1}{1- (\widehat{\theta}_1 +\widecheck{\theta}_1)}
        \\
        \widecheck T_2 (\theta) = 2 + \widecheck\theta_2
      \end{array}
    \\[30pt]
    \widehat{\Gamma}
    =
    \left[
    \begin{array}{cc}
      0
      & 1
      \\
      1
      & 0
      \\
      1
      & 0
      \\
      0
      & 0
      \\
      1
      & 0
      \\
      0
      & 0
      \\
      1
      & 0
    \end{array}
        \right]
      & \widecheck{\Gamma}
        =
        \left[
        \begin{array}{cc}
          0
          & 0
          \\
          0
          & 0
          \\
          1
          & 0
          \\
          0
          & 1
          \\
          1
          & 0
          \\
          1
          & 0
          \\
          0
          & 0
        \end{array}
            \right]
  \end{array}
\end{displaymath}

In this network, one population moves from the origin
$\widehat{\mathcal{O}}$ to the destination $\widehat{\mathcal{D}}$;
$\widehat\theta_1$ travelers move along the central route
$\widehat\gamma_{1}=(r_{2}, r_{5}, r_{7})$, while $\widehat\theta_2$
travelers move along the northern highway
$\widehat\gamma_{2} = r_{1}$. In the other population (which moves
from the origin $\widecheck{\mathcal{O}}$ to the destination
$\widecheck{\mathcal{D}}$), $\widecheck\theta_1$ travelers move along
the central route $\widecheck\gamma_{1} = (r_{3}, r_{5}, r_{6})$ and
$\widecheck\theta_2$ travelers move along the southern highway
$\widecheck\gamma_{2} = r_{4}$. Road $r_5$ is used by both
populations.

For $\Delta \in [0, 3/2\mathclose[$, a Nash equilibrium and the
corresponding travel times are:
\begin{equation}
  \label{eq:40}
  \begin{array}{r@{\;}c@{\;}l}
    \widehat\theta_\Delta
    & =
    & \left(
      \dfrac{5+5\Delta-\sqrt{\Delta^2+26\Delta+17}}{4},
      \dfrac{-1-5\Delta+\sqrt{\Delta^2+26\Delta+17}}{4}
      \right)\,,
    \\
    \widecheck\theta_\Delta
    & =
    & \left(
      \dfrac{5+\Delta-\sqrt{\Delta^2+26\Delta+17}}{4},
      \dfrac{-1-\Delta+\sqrt{\Delta^2+26\Delta+17}}{4}
      \right)\,,
    \\
    \widehat T_1 (\theta_\Delta)
    & =
    & \widehat T_2 (\theta_\Delta)
      =
      \dfrac{7-\Delta+\sqrt{\Delta^2+26\Delta+17}}{4}\,.
    \\
    \widehat T_1 (\theta_\Delta)
    & =
    & \widehat T_2 (\theta_\Delta)
      =
      \dfrac{7-\Delta+\sqrt{\Delta^2+26\Delta+17}}{4}\,.
  \end{array}
\end{equation}
Proposition~\ref{prop:unique0} ensures that
$\theta_\Delta \equiv (\widehat\theta_\Delta,
\widecheck\theta_\Delta)$ is the unique Nash equilibrium. For
$\Delta=0$, $\theta_0$ is in the interior of $S^2 \times S^2$ and
satisfies $(\widehat\theta_0)_1 + (\widecheck\theta_0)_1 < 1$, so that
all travel times are finite. As $\Delta$ increases, traveling along
$r_1$ gets less and less convenient, so that
$(\widehat\theta_\Delta)_2$ decreases, $(\widehat\theta_\Delta)_1$
increases and the equilibrium travel time also increases. In turn, for
the check population, traveling along $\widecheck\gamma_1$ is less and
less convenient, so that $\widecheck\theta_1$ decreases. When
$\Delta=1/2$, the Nash equilibrium for the check population is at
$(\widecheck\theta_{1/2})_1=0$, $(\widecheck\theta_{1/2})_2=1$, while
for the hat population it is in the interior of $S^2$, due to the
unboundedness of the travel time along $\widehat\gamma_1$.

This example shows that a delay along a route for the hat population
may radically change the optimal choice for the check population.

\subsection{Braess Paradox for \texorpdfstring{$2$}{2} Populations}
\label{subsec:braess-paradox-2}

In the present multi-population setting, phenomena like the celebrated
Braess paradox~\cite{BraessParadox} are possible and may have effects
on all the populations on the network, even in the case only one
population is affected by the introduction of a new road. The
literature on Braess paradox is extremely rich. We refer here for
instance to the textbook~\cite[\S~8.2]{MR2677125} for a general
introduction; a stochastic study is in~\cite{Bittihn2018133}; queue
theory is employed in~\cite{Lin2009117}; a variational inequality
model is used in~\cite{Nagurney2007355} while a non-stationary
approach is considered in~\cite{MR4104792}. A different
multi-population approach is in~\cite{ijcai2019p80}.

\subsubsection{Trucks and Cars on the Same Network}
\label{subsec:trucks-cars-same}

We consider here two different populations traveling along the same
network.  The hat population consists of, say, trucks while the check
population consists of cars. The former are slow and not influenced by
the latter, which are faster and slowed down by the presence of the
former. Thus, both populations travel along the same roads, but with
different travel times.

Consider the classical network~\eqref{eq:BraessPrima} leading to
Braess paradox~\cite{BraessParadox}, where $N = 4$, $\widehat n = 2$,
$\widecheck n = 2$ and all roads are one way from left to right.
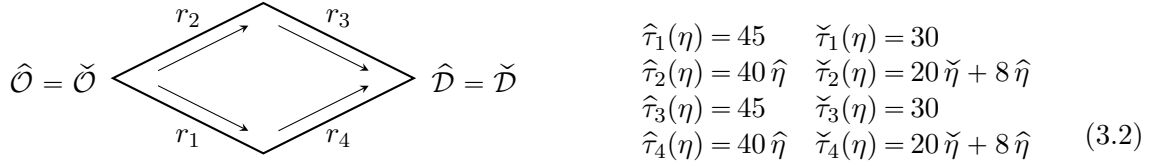
\begin{equation}
  \label{eq:BraessPrima}
  \begin{array}{@{}c@{}c@{}}
    \begin{tikzpicture}[baseline]
      \draw[thick] (-2,1) -- (0,0) -- (2,1) -- (0,2) -- cycle; %
      \draw [-stealth] (0.2,0.3) -- (1.4,0.9); %
      \draw [-stealth] (0.2,1.7) -- (1.4,1.1); %
      \draw [-stealth] (-1.4,1.1) -- (-0.2,1.7); %
      \draw [-stealth] (-1.4,0.9) -- (-0.2,0.3); %
      \draw (-2.8,1) node
      {$\widehat{\mathcal{O}}=\widecheck{\mathcal{O}}$}; %
      \draw (2.8,1) node
      {$\widehat{\mathcal{D}}=\widecheck{\mathcal{D}}$}; %
      \draw (-1,0.2) node {$r_1$};%
      \draw (1,0.2) node {$r_4$};%
      \draw (-1,1.8) node {$r_2$};%
      \draw (1,1.8) node {$r_3$};%
    \end{tikzpicture}
    & %
      \begin{array}[b]{@{}r@{\,}c@{\,}l@{}}
        \widehat\tau_1 (\eta)
        & =
        & 45
        \\
        \widehat\tau_2 (\eta)
        & =
        & 40\, \widehat\eta
        \\
        \widehat\tau_3 (\eta)
        & =
        & 45
        \\
        \widehat\tau_4 (\eta)
        & =
        & 40 \, \widehat\eta
      \end{array}
          \quad
          \begin{array}[b]{@{}r@{\,}c@{\,}l@{}}
            \widecheck\tau_1 (\eta)
            & =
            & 30
            \\
            \widecheck\tau_2 (\eta)
            & =
            & 20 \, \widecheck\eta + 8 \,\widehat\eta
            \\
            \widecheck\tau_3 (\eta)
            & =
            & 30
            \\
            \widecheck\tau_4 (\eta)
            & =
            & 20 \, \widecheck\eta + 8 \,\widehat\eta
          \end{array}
    \\[12pt]
    \begin{array}{r@{\,}c@{\,}l}
      \widehat\gamma_1
      & =
      & (r_2,r_3)
      \\
      \widehat\gamma_2
      & =
      & (r_1,r_4)
    \end{array}
        \quad
        \begin{array}{r@{\,}c@{\,}l}
          \widecheck\gamma_1
          & =
          & (r_2,r_3)
          \\
          \widecheck\gamma_2
          & =
          & (r_1,r_4)
        \end{array}
        & %
          \begin{array}{r@{\,}c@{\,}l}
            \widehat T_1 (\theta)
            & =
            & 45 + 40 \, \widehat\theta_1
            \\
            \widehat T_2 (\theta)
            & =
            & 45 + 40 \, \widehat\theta_2
          \end{array}
              \quad
              \begin{array}{r@{\,}c@{\,}l@{}}
                \widecheck T_1 (\theta)
                & =
                & 30 + 20 \, \widecheck\theta_1 + 8 \, \widehat \theta_1
                \\
                \widecheck T_2 (\theta)
                & =
                & 30 + 20 \, \widecheck\theta_2 + 8 \, \widehat \theta_2 .
              \end{array}
  \end{array}
  \!\!\!\!\!\!\!\!\!  \!\!\!\!\!\!\!\!\!
\end{equation}
Straightforward computations show that a Nash equilibrium and the
corresponding route travel times are
\begin{displaymath}
  \theta^* \equiv
  \left(\left(\frac12, \frac12\right), \left(\frac12,\frac12\right)\right)
  \qquad
  \begin{array}{r@{\;}c@{\;}l}
    \widehat T_1 (\theta^*) = \widehat T_2 (\theta^*)
    & =
    & 65\,;
    \\
    \widecheck T_1 (\theta^*) = \widecheck T_2 (\theta^*)
    & =
    & 44 \,.
  \end{array}
\end{displaymath}
The uniqueness of this Nash equilibrium follows from
Proposition~\ref{prop:unique0}. Then, we add a new road, $r_5$, as
in~\eqref{eq:Braess}, characterized by a negligible travel time. Thus
we have
\begin{equation}
  \label{eq:Braess}
  \begin{array}{@{}c@{}}
    \begin{tikzpicture}[baseline]
      \draw[thick] (-2,1) -- (0,0) -- (0,2) -- cycle; %
      \draw[thick] (2,1) -- (0,0) -- (0,2) -- cycle; %
      \draw [-stealth] (0.2,0.3) -- (1.4,0.9); %
      \draw [-stealth] (0.2,1.7) -- (1.4,1.1); %
      \draw [-stealth] (-1.4,1.1) -- (-0.2,1.7); %
      \draw [-stealth] (-1.4,0.9) -- (-0.2,0.3); %
      \draw [-stealth] (-0.2,1.4) -- (-0.2,0.6); %
      \draw (-2.8,1) node
      {$\widehat{\mathcal{O}}=\widecheck{\mathcal{O}}$}; %
      \draw (2.8,1) node
      {$\widehat{\mathcal{D}}=\widecheck{\mathcal{D}}$}; %
      \draw (-1,0.2) node {$r_1$};%
      \draw (0.3,1) node {$r_5$};%
      \draw (1,0.2) node {$r_4$};%
      \draw (-1,1.8) node {$r_2$};%
      \draw (1,1.8) node {$r_3$};%
    \end{tikzpicture}
    \qquad
    \begin{array}[b]{r@{\,}c@{\,}l}
      \widehat\tau_1 (\eta)
      & =
      & 45
      \\
      \widehat\tau_2 (\eta)
      & =
      & 40\, \widehat\eta
      \\
      \widehat\tau_3 (\eta)
      & =
      & 45
      \\
      \widehat\tau_4 (\eta)
      & =
      & 40 \, \widehat\eta
      \\
      \widehat\tau_5 (\eta)
      & =
      & 0
    \end{array}
        \qquad
        \begin{array}[b]{r@{\,}c@{\,}l}
          \widecheck\tau_1 (\eta)
          & =
          & 30
          \\
          \widecheck\tau_2 (\eta)
          & =
          & 20 \, \widecheck\eta + 8 \,\widehat\eta
          \\
          \widecheck\tau_3 (\eta)
          & =
          & 30
          \\
          \widecheck\tau_4 (\eta)
          & =
          & 20 \, \widecheck\eta + 8 \,\widehat\eta
          \\
          \widecheck\tau_5 (\eta)
          & =
          & 0
        \end{array}
    \\[6pt]
    \begin{array}{r@{\,}c@{\,}l}
      \widehat\gamma_1
      & =
      & (r_2,r_3)
      \\
      \widehat\gamma_2
      & =
      & (r_1,r_4)
      \\
      \widehat\gamma_3
      & =
      & (r_2, r_5, r_4)
      \\[6pt]
      \widecheck\gamma_1
      & =
      & (r_2,r_3)
      \\
      \widecheck\gamma_2
      & =
      & (r_1,r_4)
      \\
      \widecheck\gamma_2
      & =
      & (r_2, r_5, r_4)
    \end{array}
        \qquad\quad
        \begin{array}{@{}r@{\,}c@{\,}l@{}}
          \widehat T_1 (\theta)
          & =
          & 45 + 40 \, (\widehat\theta_1 + \widehat\theta_3)
          \\
          \widehat T_2 (\theta)
          & =
          & 45 + 40 \, (\widehat\theta_2 + \widehat\theta_3)
          \\
          \widehat T_3 (\theta)
          & =
          & 40 (\widehat\theta_1+\widehat\theta_3)
            + 40 (\widehat\theta_2+\widehat\theta_3)
          \\[6pt]
          \widecheck T_1 (\theta)
          & =
          & 30 + 20 \, (\widecheck\theta_1+ \widecheck\theta_3)
            + 8 (\widehat \theta_1 + \widehat\theta_3)
          \\
          \widecheck T_2 (\theta)
          & =
          & 30 + 20 \, (\widecheck\theta_2 + \widecheck\theta_3)
            + 8 (\widehat \theta_2+ \widecheck\theta_3)
          \\
          \widecheck T_3 (\theta)
          & =
          & 20 (\widecheck\theta_1+\widecheck\theta_3)
            + 8 (\widehat\theta_1+\widehat\theta_3)
            + 20 (\widecheck\theta_2+\widecheck\theta_3)
            + 8 (\widehat\theta_2+\widehat\theta_3) \,.
        \end{array}
  \end{array}
\end{equation}
A Nash equilibrium and the corresponding route travel times are
\begin{displaymath}
  \theta^* \equiv
  \left((0, 0, 1), (0, 0, 1)\right)
  \qquad
  \begin{array}{r@{\;}c@{\;}l@{\qquad}r@{\;}c@{\;}l@{\qquad}r@{\;}c@{\;}l}
    \widehat T_1 (\theta^*)
    & =
    & 95
    & \widehat T_2 (\theta^*)
    & =
    & 95
    & \widehat T_3 (\theta^*)
    & =
    & 80
    \\
    \widecheck T_1 (\theta^*)
    & =
    & 58
    & \widecheck T_2 (\theta^*)
    & =
    & 58
    & \widecheck T_3 (\theta^*)
    & =
    & 56 \,.
  \end{array}
\end{displaymath}
The uniqueness of this Nash equilibrium directly follows form
Proposition~\ref{prop:unique0}.

The insertion of the new road $r_5$ for both populations causes an
increase in the travel time at the unique global Nash equilibrium for
all network users.

\subsubsection{Braess Paradox Induced by a Different Population}
\label{subsec:braess-parad-induc}

In this example two different populations, the hat and the check
population, move from different origins to the same destination. Our
aim is to show that the addition of a road, aimed at helping the hat
population, has a negative effect on both the hat and the check
population.

In more details, consider the situation~\eqref{eq:49}, where $N=5$,
$\widehat n = 2$, $\widecheck n = 2$ and all roads are one way
downwards.  In the hat population, $\widehat\theta_1$ travelers move
from the origin $\widehat{\mathcal{O}}$ to the destination
$\widehat{\mathcal{D}}$ along the highway
$\widehat\gamma_{1} = r_{1}$, while $\widehat\theta_2$ move along
$\widehat\gamma_{2}=(r_{2}, r_{5})$. In the check population,
$\widecheck\theta_1$ travelers move from the origin
$\widecheck{\mathcal{O}}$ to the destination $\widecheck{\mathcal{D}}$
along $\widecheck\gamma_{1} = (r_{3}, r_{5})$, while
$\widecheck\theta_2$ move along the highway
$\widecheck\gamma_{2} = r_{4}$. Road $r_5$ is used by both
populations. Thus
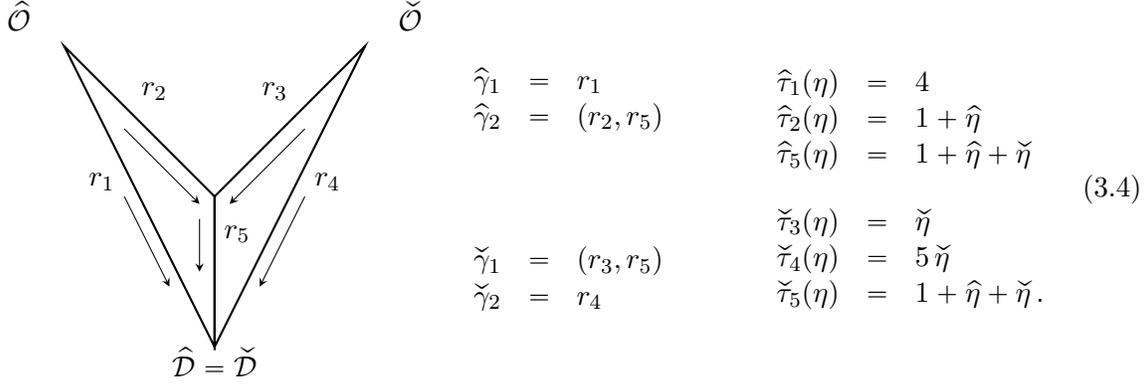
\begin{equation}
  \label{eq:49}
  \begin{tikzpicture}[baseline]
    \draw[thick] (-2,2) -- (0,0) -- (0,-2) -- cycle; %
    \draw[thick] (2,2) -- (0,0) -- (0,-2) -- cycle; %
    \draw [-stealth] (-1.2,0.9) -- (-0.2,-0.1); %
    \draw [-stealth] (1.2,0.9) -- (0.2,-0.1); %
    \draw [-stealth] (-0.2,-0.3) -- (-0.2,-1); %
    \draw [-stealth] (-1.2,0) -- (-0.6,-1.2); %
    \draw [-stealth] (1.2,0) -- (0.6,-1.2); %
    \draw (-2.6,2.4) node {$\widehat{\mathcal{O}}$}; %
    \draw (2.6,2.4) node {$\widecheck{\mathcal{O}}$}; %
    \draw (0,-2.2) node
    {$\widehat{\mathcal{D}} = \widecheck{\mathcal{D}}$}; %
    \draw (-1.5,0.2) node {$r_1$};%
    \draw (0.3,-0.5) node {$r_5$};%
    \draw (1.5,0.2) node {$r_4$};%
    \draw (-0.8,1.4) node {$r_2$};%
    \draw (0.8,1.4) node {$r_3$};%
  \end{tikzpicture}
  \quad
  \begin{array}{rcl}
    \widehat\gamma_1
    & =
    & r_1
    \\
    \widehat\gamma_2
    & =
    & (r_2,r_5)
    \\
    \\
    \\
    \\
    \widecheck\gamma_1
    & =
    & (r_3,r_5)
    \\
    \widecheck\gamma_2
    & =
    & r_4
  \end{array}
  \qquad\quad
  \begin{array}{rcl}
    \widehat\tau_1 (\eta)
    & =
    & 4
    \\
    \widehat\tau_2 (\eta)
    & =
    & 1+\widehat\eta
    \\
    \widehat\tau_5 (\eta)
    & =
    & 1 + \widehat\eta + \widecheck\eta
    \\[12pt]
    \widecheck\tau_3 (\eta)
    & =
    & \widecheck\eta
    \\
    \widecheck\tau_4 (\eta)
    & =
    & 5\, \widecheck\eta
    \\
    \widecheck\tau_5 (\eta)
    & =
    & 1 + \widehat\eta + \widecheck\eta \,.
  \end{array}
\end{equation}
The route travel times are
\begin{displaymath}
  \begin{array}{rcl}
    \widehat T_1 (\theta)
    & =
    & 4
    \\
    \widehat T_2 (\theta)
    & =
    & 2 + 2\widehat \theta_2+ \widecheck \theta_1
  \end{array}
  \qquad
  \begin{array}{rcl}
    \widecheck T_1 (\theta)
    & =
    & 1 + \widehat\theta_2 + 2 \widecheck\theta_1
    \\
    \widecheck T_2 (\theta)
    & =
    & 5 \, \widecheck \theta_2
  \end{array}
\end{displaymath}
so that a Nash equilibrium and its corresponding travel times are
\begin{equation}
  \label{eq:19}
  \begin{array}{rcl}
    \widehat\theta^*
    & =
    & (0,1)
    \\
    \widecheck\theta^*
    & =
    & (5/6,1/6)
  \end{array}
  \qquad
  \begin{array}{rcl}
    \widehat T_1 (\theta^*)
    & =
    & 4
    \\
    \widecheck T_1 (\theta^*)
    & =
    & 5/6
  \end{array}
  \qquad
  \begin{array}{rcl}
    \widehat T_2 (\theta^*)
    & =
    & 19/6
    \\
    \widecheck T_2 (\theta^*)
    & =
    & 5/6 \,.
  \end{array}
\end{equation}
Proposition~\ref{prop:unique0} ensures the uniqueness of this
equilibrium.

Introduce now a new road $r_6$ from $\widehat{\mathcal{O}}$ to
$\widecheck{\mathcal{O}}$, the new route $\widehat \gamma_3$ and the
travel time $\widehat\tau_6$, so that now $N=6$, $\widehat n = 3$ and
$\widecheck n = 2$. Modify accordingly the travel time
$\widehat\tau_6$ as follows:
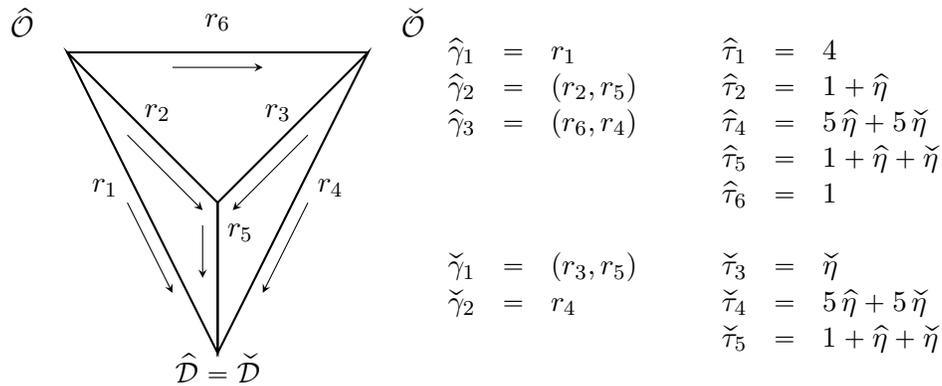
\begin{displaymath}
  \begin{tikzpicture}[baseline]
    \draw[thick] (-2,2) -- (0,0) -- (0,-2) -- cycle; %
    \draw[thick] (-2,2) -- (2,2);%
    \draw[thick] (2,2) -- (0,0) -- (0,-2) -- cycle; %
    \draw [-stealth] (-1.2,0.9) -- (-0.2,-0.1); %
    \draw [-stealth] (1.2,0.9) -- (0.2,-0.1); %
    \draw [-stealth] (-0.2,-0.3) -- (-0.2,-1); %
    \draw [-stealth] (-1.2,0) -- (-0.6,-1.2); %
    \draw [-stealth] (1.2,0) -- (0.6,-1.2) ; %
    \draw [-stealth] (-0.6,1.8) -- (0.6,1.8); %
    \draw (-2.6,2.4) node {$\widehat{\mathcal{O}}$}; %
    \draw (2.6,2.4) node {$\widecheck{\mathcal{O}}$}; %
    \draw (0,-2.2) node
    {$\widehat{\mathcal{D}} = \widecheck{\mathcal{D}}$}; %
    \draw (-1.5,0.2) node {$r_1$};%
    \draw (0.3,-0.4) node {$r_5$};%
    \draw (1.5,0.2) node {$r_4$};%
    \draw (-0.8,1.2) node {$r_2$};%
    \draw (0.8,1.2) node {$r_3$};%
    \draw (0,2.4) node {$r_6$};%
  \end{tikzpicture}
  \begin{array}{rcl}
    \widehat\gamma_1
    & =
    & r_1
    \\
    \widehat\gamma_2
    & =
    & (r_2,r_5)
    \\
    \widehat\gamma_3
    & =
    & (r_6,r_4)
    \\
    \\
    \\
    \\
    \widecheck\gamma_1
    & =
    & (r_3,r_5)
    \\
    \widecheck\gamma_2
    & =
    & r_4
    \\
    \\
  \end{array}
  \qquad
  \begin{array}{rcl}
    \widehat\tau_1
    & =
    & 4
    \\
    \widehat\tau_2
    & =
    & 1+\widehat\eta
    \\
    \widehat\tau_4
    & =
    & 5\, \widehat\eta + 5 \, \widecheck\eta
    \\
    \widehat\tau_5
    & =
    & 1 + \widehat\eta + \widecheck\eta
    \\
    \widehat\tau_6
    & =
    & 1
    \\
    \\
    \widecheck\tau_3
    & =
    & \widecheck\eta
    \\
    \widecheck\tau_4
    & =
    & 5\, \widehat\eta + 5 \, \widecheck\eta
    \\
    \widecheck\tau_5
    & =
    & 1 + \widehat\eta + \widecheck\eta
  \end{array}
\end{displaymath}
so that the route travel times are
\begin{equation}
  \label{eq:46}
  \begin{array}{r@{\,}c@{\,}l@{\qquad}r@{\,}c@{\,}l@{\qquad}r@{\,}c@{\,}l}
    \widehat T_1 (\theta)
    & =
    & 4
    &
      \widehat T_2 (\theta)
    & =
    & 2 + 2\widehat \theta_2+ \widecheck \theta_1
    &
      \widehat T_3 (\theta)
    & =
    & 1 + 5\, \widehat\theta_3 + 5\, \widecheck\theta_2
    \\
    \widecheck T_1 (\theta)
    & =
    & 1 + \widehat\theta_2 + 2 \widecheck\theta_1
    &
      \widecheck T_2 (\theta)
    & =
    & 5 \, \widehat \theta_3 + 5 \, \widecheck \theta_2 \,.
  \end{array}
\end{equation}
A Nash equilibrium and the corresponding travel times are
\begin{equation}
  \label{eq:3}
  \begin{array}{r@{\,}c@{\,}l}
    \widehat\theta^*
    & =
    & (1/15,2/3,4/15)
    \\
    \widecheck\theta^*
    & =
    & (2/3,1/3)
  \end{array}
  \qquad
  \begin{array}{r@{\,}c@{\,}l@{\qquad}r@{\,}c@{\,}l@{\qquad}r@{\,}c@{\,}l}
    \widehat T_1 (\theta^*)
    & =
    & 4
    & \widehat T_2 (\theta^*)
    & =
    & 4
    & \widehat T_3
    & =
    & 4 \,.
    \\
    \widecheck T_1 (\theta^*)
    & =
    & 3
    & \widecheck T_2 (\theta^*)
    & =
    & 3
  \end{array}
\end{equation}
Again, simple computations based on Proposition~\ref{prop:unique0}
ensure the uniqueness of this equilibrium.

The introduction of the new road for the hat population actually
worsens all travel times at the new Nash equilibrium
$\theta^* = \left((1/15,2/3,4/15), \, (2/3,1/3)\right)$. Indeed, these
travel times pass from $\widehat T^* = 19/6$ and
$\widecheck T^* = 5/6$, see~\eqref{eq:19}, to $\widehat T^* = 4$ and
$\widecheck T^* = 3$, see~\eqref{eq:3}.

\section{Proofs and Further Remarks}
\label{sec:technical-details}

The following notation is used throughout.
$\reali_+ = \mathopen[0, +\infty\mathclose[$,
$\reali_- = \mathopen]-\infty, 0\mathclose]$. $(e_1, \ldots, e_n)$ is
the canonical basis in $\reali^n$. $\one_n$ is the column vector in
$\reali^n$ whose components are all $1$: $\one_n = \sum_{i=1}^n
e_i$. The identity matrix of order $n$ is $\Id_n$. Occasionally, we
write $\zero_n$ for the null vector in $\reali^n$. The simplex $S^n$
is defined as
$S^n \coloneq \left\{\theta \in \reali_+^n \colon \sum_{i=1}^n
  \theta_i = 1 \right\}$. The cardinality of the set $I$ is
$\sharp I$.

\goodbreak\medskip

The next result is a simple computation repeatedly used in the sequel.

\begin{lemma}
  \label{lem:diff}
  Let~\ref{item:E1} hold. Then, with the notation~\eqref{eq:11}, for
  all
  $\widehat\theta, \widehat\theta', \widehat\theta'' \in S^{\widehat
    n}$ and
  $\widecheck\theta, \widecheck\theta', \widecheck\theta'' \in
  S^{\widecheck n}$ the following equalities hold, whenever the terms
  in the left hand side are finite:
  \begin{equation}
    \label{eq:26}
    \!\!\!
    \begin{array}{@{}c@{}}
      \widehat{T} (\widehat\theta'',\widecheck\theta) - \widehat{T} (\widehat\theta', \widecheck\theta)
      =
      \widehat\Gamma^\intercal \, \widehat Q \; \widehat\Gamma \, (\widehat\theta'' - \widehat\theta')
      \mbox{ where }
      \widehat Q
      =
      \int_0^1
      D_{\widehat\eta} \widehat\tau
      \left(\widehat\Gamma (s\, \widehat\theta'' + (1-s) \widehat\theta'),
      \widecheck\Gamma\,\widecheck\theta\right) \d{s};
      \\
      \widehat{T} (\widehat\theta,\widecheck\theta'') - \widehat{T} (\widehat\theta, \widecheck\theta')
      =
      \widehat\Gamma^\intercal \, \widehat P \; \widecheck\Gamma \, (\widecheck\theta'' - \widecheck\theta')
      \mbox{ where }
      \widehat P
      =
      \int_0^1
      D_{\widecheck\eta} \widehat\tau
      \left(\widehat\Gamma \widehat\theta,
      \widecheck\Gamma(s\, \widecheck\theta'' + (1-s) \widecheck\theta')\right) \d{s};
      \\
      \widecheck{T} (\widehat\theta,\widecheck\theta'') - \widecheck{T} (\widehat\theta, \widecheck\theta')
      =
      \widecheck\Gamma^\intercal \, \widecheck Q \; \widecheck\Gamma \, (\widecheck\theta'' - \widecheck\theta')
      \mbox{ where }
      \widecheck Q
      =
      \int_0^1
      D_{\widecheck\eta}\widecheck\tau\left(\widecheck\Gamma \, \widehat\theta,\widecheck\Gamma (s\, \widecheck\theta'' + (1-s) \widecheck\theta')\right) \d{s};
      \\
      \widecheck{T} (\widehat\theta'',\widecheck\theta) - \widecheck{T} (\widehat\theta', \widecheck\theta)
      =
      \widecheck\Gamma^\intercal \, \widecheck P \; \widehat\Gamma \, (\widehat\theta'' - \widehat\theta')
      \mbox{ where }
      \widecheck P
      =
      \int_0^1
      D_{\widehat\eta} \widecheck\tau\left(\widehat\Gamma (s\, \widehat\theta'' + (1-s) \widehat\theta'), \widecheck\Gamma\,\widecheck\theta\right) \d{s}.
    \end{array}
  \end{equation}
  Moreover, $\widehat Q, \widehat P, \widecheck Q$ and $\widecheck P$
  are diagonal $N\times N$ matrices with non-negative entries.
\end{lemma}

\begin{proof}
  Consider the first $2$ lines in~\eqref{eq:26} Standard calculus
  applied to~\eqref{eq:11} leads to
  \begin{eqnarray*}
    \widehat{T} (\widehat\theta'',\widecheck\theta) - \widehat{T}
    (\widehat\theta',\widecheck\theta)
    & =
    & \widehat\Gamma^\intercal
      \left( \widehat\tau (\widehat\Gamma \widehat\theta'',
      \widecheck\Gamma \,\widecheck\theta)
      - \widehat\tau (\widehat\Gamma \widehat\theta',
      \widecheck\Gamma \,\widecheck\theta)
      \right)
    \\
    & =
    & \widehat \Gamma^\intercal \int_0^1 D_{\widehat\eta}
      \widehat\tau \left(\widehat \Gamma (s\, \widehat\theta'' + (1-s)
      \widehat\theta',\widecheck\Gamma\, \widecheck\theta)\right) \d{s} \; \widehat
      \Gamma \, (\widehat\theta'' - \widehat\theta') \,.
    \\
    \widehat{T} (\widehat\theta,\widecheck\theta'') - \widehat{T} (\widehat\theta, \widecheck\theta')
    & =
    & \widehat\Gamma^\intercal \left(
      \widehat\tau (\widehat\Gamma\, \widehat\theta,\widecheck\Gamma \,\widecheck\theta'')
      -
      \widehat\tau (\widehat\Gamma \, \widehat\theta, \widecheck\Gamma \,\widecheck\theta')
      \right)
    \\
    & =
    & \widehat\Gamma^\intercal
      \int_0^1 D_{\widecheck\eta} \widehat\tau
      \left(\widehat\Gamma \, \widehat\theta,
      \widecheck\Gamma (s \widecheck\theta'' + (1-s)\widecheck\theta')\right)\d{s}
      \; \widecheck\Gamma (\widecheck\theta'' - \widecheck\theta') \,.
  \end{eqnarray*}
  Moreover, \eqref{eq:10} ensures that the two $N \times N $ matrices
  $\widehat Q$ and $\widehat P$ are diagonal and, by~\ref{item:E1},
  \begin{eqnarray*}
    \widehat Q_{hh}
    & =
    & \int_0^1 \partial_{\widehat\eta_h} \widehat\tau_h \left(\widehat \Gamma
      (s\, \widehat\theta'' + (1-s) \widehat\theta',\widehat \Gamma \, \widecheck\theta)\right) \d{s}
      \geq 0
    \\
    \widehat P_{hh}
    & =
    & \int_0^1 \partial_{\widecheck\eta_h} \widehat\tau_h \left(\widehat \Gamma \, \widehat\theta, \widecheck\Gamma
      (s\, \widecheck\theta'' + (1-s) \widecheck\theta')\right) \d{s}
      \geq 0
  \end{eqnarray*}
  proving the first $2$ lines in~\eqref{eq:26}.  The last $2$ lines
  of~\eqref{eq:26} are proved analogously.
\end{proof}

The following observation is of use below.

\begin{remark}
  \label{rem:perm}
  The matrices
  $\widehat \Gamma^\intercal \, \widehat Q \, \widehat \Gamma$,
  $\widehat\Gamma \, \widehat P \, \widecheck \Gamma$,
  $\widecheck \Gamma^\intercal \, \widecheck Q \, \widecheck \Gamma$
  and $\widecheck\Gamma \, \widecheck P \, \widehat \Gamma$
  in~\eqref{eq:26} are invariant with respect to any renumbering of
  roads or routes.
\end{remark}

Indeed, renumbering the roads amounts to a permutation of the rows of
$\widehat \Gamma$ and of the rows and the columns of $\widehat Q$ and
$\widehat{P}$ i.e., we have $\widehat\Gamma \to \Pi\,\widehat\Gamma$
and $\widehat Q \to \Pi \, \widehat Q \, \Pi^{-1}$, for a suitable
$N\times N$ permutation matrix $\Pi$. Then, since $\Pi$ is orthogonal,
$\Gamma^\intercal \widehat Q \, \Gamma \to (\Pi\,\Gamma)^\intercal
(\Pi\, \widehat Q \, \Pi^{-1}) (\Pi\,\Gamma) = \Gamma^\intercal
\widehat Q \, \Gamma$, and similarly for the other matrices.

\begin{lemma}
  \label{lem:gamma}
  Let $\Gamma$ be an $N\times n$ matrix whose entries are $0$ or
  $1$. Let $M$ be an $N \times N$ diagonal matrix with non-negative
  entries. Then, for all $i,j \in \{1, \ldots, n\}$
  \begin{displaymath}
    \left(\Gamma^\intercal \, M \, \Gamma\right) _{jj} \geq \left(\Gamma^\intercal \, M \, \Gamma\right)_{ji}
  \end{displaymath}
\end{lemma}

\begin{proof}
  Let $i,j \in \{1, \ldots, n\}$ and $k \in \{1, \ldots, N\}$. Then,
  $\left(\Gamma^\intercal \, M \, \Gamma\right)_{ij} = \sum_{k=1}^N
  \Gamma_{ki} \, \left(M \, \Gamma\right)_{kj} = \sum_{k=1}^N
  \Gamma_{ki} \, M_{kk} \, \Gamma_{kj}$.  Recalling that
  $(\Gamma_{kj})^2 = \Gamma_{kj}$, we have
  \begin{displaymath}
    \left(\Gamma^\intercal \, M \, \Gamma\right)_{jj}
    =
    \sum_{k=1}^N \Gamma_{kj} \, M_{kk}
    \qquad \mbox{ and } \qquad
    \left(\Gamma^\intercal \, M \, \Gamma\right)_{ji}
    =
    \sum_{k=1}^N \Gamma_{ki} \, \Gamma_{kj} \, M_{kk}
  \end{displaymath}
  showing that the latter sum contains either part or all of the terms
  in the former sum. Since all summands are non-negative, the proof
  follows.
\end{proof}

\begin{proofof}{Lemma~\ref{lem:mono}}
  Straightforward computations yield:
  \begin{flalign*}
    \widehat{T}_i (\theta) - \widehat{T}_i (\widehat\theta -
    \epsilon\, e_i + \epsilon \, e_j,\widecheck\theta) & =
    e_i^\intercal \; \widehat\Gamma^\intercal \, \widehat Q \;
    \widehat\Gamma \; (\epsilon\, e_i - \epsilon \, e_j) & [\mbox{By
      Lemma~\ref{lem:diff}}]
    \\
    & = \epsilon \left( e_i^\intercal \; \widehat\Gamma^\intercal \,
      \widehat Q \; \widehat\Gamma \, e_i - e_i^\intercal \;
      \widehat\Gamma^\intercal \, \widehat Q \; \widehat\Gamma \, e_j
    \right)
    \\
    & = \epsilon \left( (\widehat\Gamma^\intercal \, \widehat Q \;
      \widehat\Gamma)_{ii} - (\widehat\Gamma^\intercal \, \widehat Q
      \; \widehat\Gamma)_{ij} \right)
    \\
    & \geq 0 & [\mbox{By Lemma~\ref{lem:gamma}}]
  \end{flalign*}
  proving~\eqref{eq:21}. Inequality~\eqref{eq:27} is proved similarly
  and~\eqref{eq:39} follows by continuity.
\end{proofof}

The next Lemma provides a further general monotonicity property.

\begin{lemma}
  \label{lem:monotone}
  Let~\ref{item:E1} hold. Fix $j \in \{1, \ldots, \widehat n\}$,
  $i \in \{1, \ldots, \widecheck n\}$ and any
  $\theta \in S^{\widehat n}\times S^{\widecheck n}$. Then, both the
  maps
  \begin{displaymath}
    \begin{array}{c@{\,}c@{\,}c@{\,}c@{\, }c}
      \widehat\Upsilon_j
      & \colon
      & [0,1]
      & \to
      & \reali_+
      \\
      &
      & \sigma
      & \mapsto
      & \widehat{T}_j\left(\sigma \, e_j + (1-\sigma) \widehat\theta, \widecheck\theta\right)
    \end{array}
    \qquad
    \begin{array}{c@{\,}c@{\,}c@{\,}c@{\, }c}
      \widecheck\Upsilon_i
      & \colon
      & [0,1]
      & \to
      & \reali_+
      \\
      &
      & \sigma
      & \mapsto
      & \widecheck{T}_i\left(\widehat\theta, \sigma \, e_i + (1-\sigma) \widecheck\theta\right)
    \end{array}
  \end{displaymath}
  are weakly increasing (i.e., non-decreasing).
\end{lemma}

\begin{proof}
  Fix $\sigma',\sigma''\in [0,1]$ with $\sigma' < \sigma''$. Write
  $\widehat\theta' = \sigma' \, e_j + (1-\sigma') \widehat\theta$ and
  $\widehat\theta'' = \sigma'' \, e_j + (1-\sigma'') \widehat\theta$.
  Use~\eqref{eq:11} to compute:
  \begin{flalign*}
    \widehat\Upsilon_j (\sigma'') - \widehat\Upsilon_j (\sigma') & =
    \widehat{T}_j (\widehat\theta'', \widecheck\theta) - \widehat{T}_j
    (\widehat\theta', \widecheck\theta) & [\mbox{Definition of
      $\widehat\Upsilon_j$}]
    \\
    & = e_j^\intercal \left(\widehat{T}
      (\widehat\theta'',\widecheck\theta) - \widehat T
      (\widehat\theta',\widecheck\theta)\right)
    \\
    & = {e_j}^\intercal \; \widehat\Gamma^\intercal \; \widehat Q \;
    \widehat\Gamma \; (\widehat\theta''-\widehat\theta') &
    [\mbox{Lemma~\ref{lem:diff}}]
    \\
    & = {e_j}^\intercal \; \widehat\Gamma^\intercal \; \widehat Q \;
    \widehat\Gamma \; (\sigma''-\sigma') (e_j-\widehat\theta) &
    [\mbox{Definition of $\widehat\theta',\widehat\theta''$}]
    \\
    & = (\sigma''-\sigma') \; {e_j}^\intercal \;
    \widehat\Gamma^\intercal \; \widehat Q \; \widehat\Gamma \;
    (e_j-\widehat\theta)
    \\
    & = (\sigma''-\sigma') \; \sum_{i=1}^n \widehat\theta_i \left(
      (\widehat\Gamma^\intercal \; \widehat Q \; \widehat\Gamma)_{jj}
      - (\widehat\Gamma^\intercal \; \widehat Q \;
      \widehat\Gamma)_{ji} \right)
    \\
    & \geq 0 \,. & [\mbox{by Lemma~\ref{lem:gamma}}]
  \end{flalign*}
  The case of $\widecheck\Upsilon_i$ is entirely similar and the proof
  is completed.
\end{proof}

\begin{proofof}{Theorem~\ref{thm:convex}}
  The implication from $\epsilon$--Nash equilibrium to Nash
  equilibrium under the continuity of the travel times follows by a
  limiting procedure.

  Let $\theta^*$ be a Nash equilibrium. Define
  $\bar\epsilon = \frac{1}{2} \min \left\{\widehat\theta^*_i \colon
    \widehat\theta^*_i > 0\right\}$.  Following
  Definition~\ref{def:Eq}, fix $i,j \in \{1, \ldots, n\}$ so that
  $\widehat\theta^* - \epsilon \, e_i + \epsilon \, e_j \in S^n$ for
  $\epsilon \in [0, \bar\epsilon]$. We thus have
  $\widehat\theta^*_i >0$ so that
  $\widehat{T}_i (\theta^*) = \widehat{T}_j (\theta^*)$ if
  $\widehat\theta^*_j > 0$, or
  $\widehat{T}_i (\theta^*) \leq \widehat{T}_j (\theta^*)$ by
  Definition~\ref{def:candiate} if $\widehat\theta^*_j = 0$. In both
  cases, using also~\eqref{eq:11},
  \begin{eqnarray}
    \nonumber
    \widehat{T}_j (\widehat\theta^* - \epsilon \, e_i + \epsilon\, e_j, \widecheck\theta^*)
    -
    \widehat{T}_i (\theta^*)
    & \geq
    & \widehat{T}_j (\widehat\theta^* - \epsilon \, e_i + \epsilon\, e_j,\widecheck\theta^*)
      - \widehat{T}_j (\widehat\theta^*,\widecheck\theta^*)
    \\
    \label{eq:37}
    & =
    & {e_j}^\intercal
      \left(
      \widehat{T} (\widehat\theta^* - \epsilon \, e_i + \epsilon\, e_j, \widecheck\theta^*)
      -
      \widehat{T} (\theta^*)
      \right) \,.
  \end{eqnarray}
  Under Assumption~\ref{item:E1}, define
  $\widehat Q = \int_0^1
  D_{\widehat\eta}\widehat\tau\left(\widehat\Gamma (s\, \widehat
    \theta^* + (1-s)
    \widehat\theta^*),\widecheck\Gamma\,\widecheck\theta^*\right)
  \d{s}$ and use Lemma~\ref{lem:diff}:
  \begin{eqnarray}
    \nonumber
    \widehat{T}_j (\widehat\theta^* - \epsilon \, e_i + \epsilon\, e_j, \widecheck\theta^*)
    -
    \widehat{T}_i (\widehat\theta^*,\widecheck\theta^*)
    & \geq
    & {e_j}^\intercal
      \left(
      \widehat{T} (\widehat\theta^* - \epsilon \, e_i + \epsilon\, e_j, \widecheck\theta^*)
      -
      \widehat{T} (\theta^*)
      \right)
    \\
    \label{eq:41}
    & =
    & {e_j}^\intercal \, \widehat\Gamma^\intercal \, \widehat Q \, \widehat\Gamma (-\epsilon \, e_i + \epsilon \, e_j) \,.
  \end{eqnarray}

  On the other hand, under Assumption~\ref{item:E2}, we have
  \begin{equation}
    \label{eq:13}
    \forall h \in \{1, \ldots, N\} \;\,
    \forall \eta_o \in [0,1] \;\,
    \exists\, \ell^h_{\eta_o} \in \reali_+ \colon
    \forall \eta \in [0,1] \;\,
    \widehat\tau_h (\eta,\widecheck\theta^*)
    \geq
    \widehat\tau_h (\eta_o,\widecheck\theta^*) + \ell^h_{\eta_o} (\eta - \eta_o) \,.
  \end{equation}
  In the terminology of convex analysis, $\ell^h_{\eta_o}$ is a
  \emph{subgradient} of
  $\eta \mapsto \widehat\tau_h (\eta,\widecheck\theta^*)$ at $\eta_o$.
  Introduce the $N\times N$ diagonal matrix $L$ with entries
  $L_{hh} = \ell^h_{\widehat\theta^*_j}$, for $h=1, \ldots,N$, as
  defined in~\eqref{eq:13}, so that $L_{hh} \geq 0$. All entries in
  $e_j$, $\Gamma^\intercal$ and $\Gamma$ are non-negative, actually
  either $0$ or $1$. Continue from~\eqref{eq:37} using~\eqref{eq:11}:
  \begin{eqnarray}
    \nonumber
    \widehat{T}_j (\widehat\theta^* - \epsilon \, e_i + \epsilon\, e_j, \widecheck\theta^*)
    -
    \widehat{T}_i (\widehat\theta^*,\widecheck\theta^*)
    & \geq
    & {e_j}^\intercal
      \left(
      \widehat{T} (\widehat\theta^* - \epsilon \, e_i + \epsilon\, e_j, \widecheck\theta^*)
      -
      \widehat{T} (\theta^*)
      \right)
    \\
    \nonumber
    & \geq
    & {e_j}^\intercal \, \widehat\Gamma^\intercal
      \left(
      \widehat\tau (\widehat\Gamma \widehat\theta^*, \widecheck\Gamma \widecheck\theta^*)
      +
      L\, \Gamma\left(- \epsilon \, e_i + \epsilon\, e_j\right)
      -
      \widehat\tau (\widehat\Gamma \widehat\theta^*, \widecheck\Gamma \widecheck\theta^*)
      \right)
    \\
    \label{eq:42}
    & =
    & {e_j}^\intercal \, \widehat\Gamma^\intercal \, L\, \widehat\Gamma\left(- \epsilon \, e_i + \epsilon\, e_j\right) \,.
  \end{eqnarray}
  Both, \eqref{eq:41} and~\eqref{eq:42} yield, by
  Lemma~\ref{lem:gamma}:
  \begin{eqnarray*}
    \widehat{T}_j (\widehat\theta^* - \epsilon \, e_i + \epsilon\, e_j, \widecheck\theta^*)
    -
    \widehat{T}_i (\widehat\theta^*,\widecheck\theta^*)
    & \geq
    & {e_j}^\intercal \, \widehat\Gamma^\intercal \, L\, \widehat\Gamma\left(- \epsilon \, e_i + \epsilon\, e_j\right)
    \\
    & =
    & \epsilon \left(
      {e_j}^\intercal \, \widehat\Gamma^\intercal \, L\, \widehat\Gamma \, e_j
      -
      {e_j}^\intercal \, \widehat\Gamma^\intercal \, L\, \widehat\Gamma \, e_i
      \right)
    \\
    & =
    & \epsilon \left(
      \left(\widehat\Gamma^\intercal \, L \, \widehat\Gamma\right)_{jj}
      -
      \left(\widehat\Gamma^\intercal \, L \, \widehat\Gamma\right)_{ji}
      \right)
    \\
    & \geq
    & 0 \,.
  \end{eqnarray*}
  An entirely similar argument applies to $\widecheck{T}$. Thus,
  $\theta^*$ is an $\epsilon$--Nash equilibrium and the proof is
  completed.
\end{proofof}

\begin{proofof}{Proposition~\ref{prop:GlobNash}}
  Consider~\ref{it:N1}. Let $i,j \in \{1, \ldots, \widehat n\}$ be
  such that $\widehat\theta^*_i >0$ and $\widehat\theta^*_j>0$. Then,
  define
  $\widehat\theta^\epsilon = \widehat\theta^* - \epsilon \, e_i +
  \epsilon \, e_j$ and compute:
  \begin{flalign*}
    0 & \geq \widehat{T}_M (\theta^*) - \widehat{T}_M
    (\widehat\theta^\epsilon,\widecheck\theta^*) %
    & [\mbox{By~\eqref{eq:29}}]
    \\
    & = \sum_{\ell=1}^n \left(%
      \widehat\theta^*_\ell \, \widehat{T}_\ell (\theta^*) -
      \widehat\theta^\epsilon_\ell \, \widehat{T}_\ell
      (\widehat\theta^\epsilon,\widecheck\theta^*) \right) &
    [\mbox{By~\eqref{eq:6}}]
    \\
    & = \widehat\theta^*_i \, \widehat{T}_i (\theta^*) -
    \widehat\theta^\epsilon_i \, \widehat
    T_i(\widehat\theta^\epsilon,\widecheck\theta^*) & [\widehat{T}_i
    (\widehat \theta^\epsilon,\widecheck\theta^*) \leq \widehat{T}_i
    (\theta^*) \mbox{ by~\eqref{eq:21}}]
    \\
    & \qquad + \widehat\theta^*_j \, \widehat{T}_j (\theta^*) -
    \widehat\theta^\epsilon_j \, \widehat{T}_j
    (\widehat\theta^\epsilon,\widecheck\theta^*) & [\widehat{T}_j
    (\theta^*) = \widehat{T}_i (\theta^*) \mbox{, $\theta^*$
      equilibrium}]
    \\
    & \qquad + \sum_{\ell\neq i,j} \widehat\theta^*_\ell \left(
      \widehat T_\ell (\theta^*) - \widehat{T}_\ell
      (\widehat\theta^\epsilon,\widecheck\theta^*) \right) &
    [\widehat{T}_\ell (\widehat\theta^*,\widecheck\theta^*) = \widehat
    T_\ell(\widehat\theta^\epsilon,\widecheck\theta^*) \mbox{ by
      Def.~\ref{def:Eq}}]
    \\
    & \geq (\widehat\theta^*_i - \widehat\theta^\epsilon_i +
    \widehat\theta_j^*) \, \widehat{T}_i (\theta^*) -
    \widehat\theta^\epsilon_j \, \widehat{T}_j
    (\widehat\theta^\epsilon,\widecheck\theta^*) & [\widehat\theta^*_i
    - \widehat\theta^\epsilon_i + \widehat\theta_j^*=
    \widehat\theta^\epsilon_j]
    \\
    & = \widehat\theta^\epsilon_j \left(\widehat T_i (\theta^*) -
      \widehat T_j (\widehat \theta^\epsilon, \widecheck
      \theta^*)\right)
  \end{flalign*}
  proving that
  $\widehat T_j (\widehat \theta^* - \epsilon \, e_i + \epsilon \,
  e_j, \widecheck \theta^*) \geq \widehat T_i (\theta^*)$, so that
  $\theta^*$ is an $\epsilon$--Nash equilibrium by
  Definition~\ref{def:Nash}.

  If on the other hand $\widehat\theta^*_j = 0$, then the same
  computations as above lead to
  \begin{flalign*}
    0 & \geq \widehat\theta^*_i \, \widehat{T}_i (\theta^*) -
    \widehat\theta^\epsilon_i \, \widehat{T}_i
    (\widehat\theta^\epsilon,\widecheck\theta^*) -
    \widehat\theta^\epsilon_j \, \widehat{T}_j
    (\widehat\theta^\epsilon, \widecheck\theta^*) &
    [\widehat\theta^\epsilon_j = \epsilon]
    \\
    & = (\widehat\theta^\epsilon_i + \epsilon) \, \widehat{T}_i
    (\widehat\theta^*,\widecheck\theta^*) - \widehat\theta^\epsilon_i
    \, \widehat{T}_i (\widehat\theta^\epsilon,\widecheck\theta^*) -
    \epsilon \, \widehat{T}_j (\widehat\theta^\epsilon,
    \widecheck\theta^*)
    \\
    & = \widehat\theta^\epsilon_i \left(\widehat{T}_i (\theta^*) -
      \widehat{T}_i
      (\widehat\theta^\epsilon,\widecheck\theta^*)\right) + \epsilon
    \left(\widehat{T}_i (\theta^*) - \widehat{T}_j
      (\widehat\theta^\epsilon,\widecheck\theta^*)\right) & [\mbox{by
      Lemma~\ref{lem:mono}}]
    \\
    & \geq \epsilon \left(\widehat{T}_i (\theta^*) - \widehat{T}_j
      (\widehat\theta^\epsilon,\widecheck\theta^*)\right)
  \end{flalign*}
  which also implies
  $\widehat T_j (\widehat \theta^* - \epsilon \, e_i + \epsilon \,
  e_j, \widecheck\theta^*) \geq \widehat T_i (\theta^*)$. The proof
  of~\ref{it:N1} is completed, once the same computations are repeated
  for $\widecheck T$.

  Item~\ref{it:N2} follows, since extreme points are equilibrium
  points.
\end{proofof}

\begin{proofof}{Theorem~\ref{thm:existence}}
  Introduce the functions
  \begin{equation}
    \label{eq:40new}
    \begin{array}{ccccc}
      \bar\phi
      & \colon
      & \reali_+ \cup \{+\infty\}
      & \to
      & [0,1]
      \\
      &
      & \xi
      & \mapsto
      & \left\{
        \begin{array}{ll}
          \xi/(1-\xi)
          & \xi \in \reali_+ \,,
          \\
          1
          & \xi = +\infty \,.
        \end{array}
            \right.
      \\
      \phi
      & \colon
      & \left(\reali_+ \cup \{+\infty\}\right)^n
      & \to
      & [0,1]^n
      \\
      &
      & (\xi_1, \ldots, \xi_n)
      & \mapsto
      &\left(\bar\phi (\xi_1), \ldots, \bar\phi (\xi_n)\right)
    \end{array}
  \end{equation}
  Call $C^n_+$ the set of those $\theta \in \reali^n$ such that
  $\theta_i >0$ for at least one index $i \in \{1,\ldots,n\}$.  Define
  for all $\theta \in C_+^n$ the normalization
  $\mathcal{N} (\theta) \in S^n$ by
  $\mathcal{N} (\theta)_i = \dfrac{\max\{0,
    \theta_i\}}{\sum_{j=1}^n\max\{0, \theta_j\}}$, for
  $i \in \{1, \ldots, n\}$.

  With a slight abuse of notation, we use the same letters $\phi$ or
  $\bar\phi$ independently of whether they are defined on
  $\reali^{\widehat n} \cup \{+\infty\}$ or
  $\reali^{\widecheck n} \cup \{+\infty\}$. The same simplification
  applies to $\mathcal{N}$.

  Recalling~\eqref{eq:11}--\eqref{eq:6} and Definition~\ref{def:Eq},
  define the map
  $\mathcal{F} \colon S^{\widehat n} \times S^{\widecheck n} \to
  S^{\widehat n} \times S^{\widecheck n}$ through its components
  $\widehat{\mathcal{F}}$ and $\widecheck{\mathcal{F}}$:
  \begin{equation}
    \label{eq:15}
    \begin{array}{@{}c@{\;}c@{\;}c@{\;}c@{\;}c@{}}
      \widehat{\mathcal{F}}
      &\colon
      & S^{\widehat n} \times S^{\widecheck n}
      & \to
      & S^{\widehat n}
      \\
      &
      & \theta
      & \mapsto
      & \mathcal{N} \left(
        \widehat\theta
        -
        \lambda
        \left(
        (\phi\circ\widehat{T}) (\theta)
        - \widehat\theta^\intercal (\phi\circ{\widehat{T}}) (\theta)
        \, \one_{\widehat n}
        \right)
        \right)
      \\
      \widecheck{\mathcal{F}}
      &\colon
      & S^{\widehat n} \times S^{\widecheck n}
      & \to
      & S^{\widecheck n}
      \\
      &
      & \theta
      & \mapsto
      & \mathcal{N} \left(
        \widecheck\theta
        -
        \lambda
        \left(
        (\phi\circ\widecheck{T}) (\theta)
        - \widecheck\theta^\intercal (\phi\circ{\widecheck{T}}) (\theta)
        \, \one_{\widecheck n}
        \right)
        \right)
    \end{array}
  \end{equation}
  where
  \begin{equation}
    \label{eq:2}
    \lambda
    =
    \frac12
    \min\left\{\frac1{\widehat n}, \frac1{\widecheck n}\right\} \,.
  \end{equation}

  \paragraph{Claim~1:
    $\mathcal{F} \in \C0 (S^{\widehat n} \times S^{\widecheck
      n};S^{\widehat n} \times S^{\widecheck n})$.}

  Given the definition of $\mathcal{F}$ and $\mathcal{N}$, it is
  sufficient to verify that the denominator in the expression of each
  component of $\mathcal{N}$ does not vanish.  Proceed by
  contradiction:
  \begin{eqnarray*}
    &
    &
      \sum_{j=1}^{\widehat n} \max \left\{
      0,
      \widehat\theta_j
      - \lambda
      \left(
      (\bar\phi \circ \widehat{T}_j) (\theta) -
      \theta^\intercal
      (\phi \circ {\widehat{T}})(\theta)
      \right)
      \right\}
      =
      0
    \\
    &  \iff
    & \forall \, j = 1, \ldots,\widehat{n}
      \qquad
      \displaystyle
      \widehat\theta_j
      - \lambda
      \left(
      (\bar \phi \circ \widehat{T}_j) (\theta) -
      \theta^\intercal
      (\phi \circ {\widehat{T}})(\theta)
      \right)
      \leq
      0
    \\
    & \implies
    & \sum_{j=1}^{\widehat n}
      \left[
      \widehat\theta_j
      - \lambda
      \left(
      (\bar\phi \circ \widehat{T}_j) (\theta) -
      \theta^\intercal
      (\phi \circ {\widehat{T}})(\theta)
      \right)
      \right]
      \leq
      0
    \\
    & \iff
    & \lambda
      \sum_{j=1}^{\widehat n}
      \left(
      (\bar\phi \circ \widehat{T}_j) (\theta) -
      \theta^\intercal
      (\phi \circ {\widehat{T}})(\theta)
      \right)
      \geq
      1
    \\
    & \implies
    & \lambda
      \sum_{j=1}^{\widehat n}
      (\bar\phi \circ \widehat{T}_j) (\theta)
      \geq
      1
  \end{eqnarray*}
  which contradicts the choice~\eqref{eq:2} of $\lambda$, since
  $\sum_{j=1}^{\widehat n}(\bar\phi \circ \widehat{T}_j) (\theta) \leq
  \widehat n$. An entirely similar argument applies to the second
  component $\widecheck{\mathcal{F}}$ of $\mathcal{F}$, proving
  Claim~1.\claimend

  \goodbreak \medskip

  By Brouwer Fixed Point Theorem~\cite[Theorem~1.6.2]{MR0488102},
  $\mathcal{F}$ admits a fixed point, say $\theta^*$.

  \paragraph{Claim~2:
    $\displaystyle \widehat\theta^*_i = 0 \implies (\bar\phi \circ
    \widehat{T}_i) (\theta^*) \geq (\widehat \theta^*)^\intercal\,
    (\phi \circ {\widehat T})(\theta^*)$ and
    $\displaystyle \widecheck\theta^*_i = 0 \implies (\bar\phi \circ
    \widecheck{T}_i) (\theta^*) \geq (\widecheck \theta^*)^\intercal\,
    (\phi \circ {\widecheck T})(\theta^*)$.}

  Direct computations yield:
  \begin{eqnarray*}
    \widehat\theta^* = \widehat{\mathcal{F}} (\theta^*)
    \quad \mbox{ and } \quad
    \widehat\theta^*_i = 0
    & \implies
    & \max \left\{
      0,
      0 - \lambda
      \left(
      (\bar\phi\circ\widehat{T}_i) (\theta^*) -
      \widehat\theta^\intercal (\phi \circ {\widehat{T}})(\theta^*)
      \right)
      \right\}
      =
      0
    \\
    & \iff
    & (\bar\phi \circ \widehat{T}_i)(\theta^*)
      \geq (\theta^*)^\intercal \,
      {\widehat{T}}(\theta^*) \,.
  \end{eqnarray*}
  The second implication is identical. Claim~2 is proved.\claimend

  \paragraph{Claim~3:
    $\displaystyle \widehat\theta^*_i > 0 \implies \widehat\theta^*_i
    > \lambda \left( (\bar\phi \circ \widehat{T}_i) (\theta^*) -
      (\widehat\theta^*)^\intercal \, (\phi \circ {\widehat{T}})
      (\theta^*)\right)$ and
    $\displaystyle \widecheck\theta^*_i > 0 \implies
    \widecheck\theta^*_i > \lambda \left( (\bar\phi \circ
      \widecheck{T}_i) (\theta^*) - (\widecheck\theta^*)^\intercal \,
      (\phi \circ {\widecheck T}) (\theta^*)\right)$.}

  Similar computations give:
  \begin{eqnarray*}
    \widehat\theta^* = \widehat{\mathcal{F}} (\theta^*)
    \quad \mbox{ and } \quad
    \widehat\theta^*_i > 0
    & \implies
    & \max \left\{
      0,
      \widehat\theta^*_i - \lambda
      \left(
      (\bar\phi \circ \widehat{T}_i) (\theta^*) -
      (\widehat\theta^*)^\intercal \, (\phi \circ {\widehat{T}})
      (\theta^*)
      \right)
      \right\}
      >
      0
    \\
    & \iff
    & \widehat\theta^*_i
      >
      \lambda
      \left(
      (\bar\phi \circ \widehat{T}_i) (\theta^*) -
      (\widehat\theta^*)^\intercal \, (\phi \circ {\widehat{T}})
      (\theta^*)
      \right) \,,
  \end{eqnarray*}
  proving Claim~3, since the other implication is identical.\claimend

  \paragraph{Claim~4: $\displaystyle \widehat\theta^*_i > 0$ and
    $\widehat\theta^*_j > 0 \implies \widehat{T}_i (\theta^*) =
    \widehat{T}_j (\theta^*)$.}

  By the above claims:
  \begin{flalign*}
    \widehat\theta^*_i & = \dfrac{\max\left\{ 0, \widehat\theta^*_i -
        \lambda \left((\bar\phi \circ \widehat{T}_i) (\theta^*) -
          (\widehat\theta^*)^\intercal \, (\phi \circ {\widehat{T}})
          (\theta^*) \right) \right\}}{%
      \sum\limits_{k=1}^n \max \left\{ 0, \widehat\theta^*_k - \lambda
        \left( (\bar\phi \circ \widehat{T}_k) (\theta^*) -
          (\widehat\theta^*)^\intercal \, (\phi \circ {\widehat{T}})
          (\theta^*) \right) \right\} } & [\mbox{$\theta^*$ is a fixed
      point}]
    \\
    & = \dfrac{%
      \widehat\theta^*_i - \lambda \left( (\bar\phi \circ
        \widehat{T}_i) (\theta^*) - (\widehat\theta^*)^\intercal \,
        (\phi \circ {\widehat{T}}) (\theta^*) \right) }{%
      \sum\limits\limits_{k \colon \widehat\theta^*_k>0} \max \left\{
        0, \widehat\theta^*_k - \lambda \left( (\bar\phi \circ
          \widehat{T}_k) (\theta^*) - (\widehat\theta^*)^\intercal \,
          (\phi \circ {\widehat{T}}) (\theta^*) \right) \right\} } &
    [\widehat\theta^*_i > 0 \mbox{ and Claim~3}]
    \\
    & = \dfrac{%
      \widehat\theta^*_i - \lambda \left( (\bar\phi \circ
        \widehat{T}_i) (\theta^*) - (\widehat\theta^*)^\intercal \,
        (\phi \circ {\widehat{T}}) (\theta^*) \right) }{%
      \sum\limits\limits_{k \colon \widehat\theta^*_k>0} \left(
        \widehat\theta^*_k - \lambda \left( (\bar\phi \circ
          \widehat{T}_k) (\theta^*) - (\widehat\theta^*)^\intercal \,
          (\phi \circ {\widehat{T}}) (\theta^*) \right) \right) } &
    [\widehat\theta^*_k > 0 \mbox{ and Claim~3}]
    \\
    & = \dfrac{ \widehat\theta^*_i - \lambda \left( (\bar\phi \circ
        \widehat{T}_i) (\theta^*) - (\widehat\theta^*)^\intercal \,
        (\phi \circ {\widehat{T}}) (\theta^*) \right) }{%
      1 - \lambda \sum\limits\limits_{k \colon \widehat\theta^*_k>0}
      \left( (\bar\phi \circ \widehat{T}_k) (\theta^*) -
        (\widehat\theta^*)^\intercal \, (\phi \circ {\widehat{T}})
        (\theta^*) \right) } & [\sum\limits_{k \colon
      \widehat\theta^*_k>0} \widehat\theta^*_k =1]
  \end{flalign*}
  so that
  \begin{equation}
    \label{eq:4}
    (\bar\phi \circ
    \widehat{T}_i) (\theta^*) - (\widehat\theta^*)^\intercal \,
    (\phi \circ {\widehat{T}}) (\theta^*)
    =
    \widehat\theta_i^* \,
    \sum\limits_{k \colon \widehat\theta^*_k>0}
    \left(
      (\bar\phi \circ
      \widehat{T}_k) (\theta^*) - (\widehat\theta^*)^\intercal \,
      (\phi \circ {\widehat{T}}) (\theta^*)    \right) \,.
  \end{equation}
  Repeating the same procedure with $\widehat\theta^*_j$, we have that
  for all $j$ such that $\widehat\theta^*_j > 0$, the differences
  $(\bar\phi \circ T_i) (\theta^*) - (\widehat\theta^*)^\intercal
  (\phi\circ T) (\theta^*)$ all have the same sign. Since
  $(\widehat\theta^*)^\intercal (\phi\circ T) (\theta^*)$ is a convex
  combination of these $(\bar\phi \circ T_j) (\theta^*)$, we have that
  for all $j$ such that $\widehat\theta^*_j > 0$,
  $(\bar\phi \circ T_j) (\theta^*) = (\widehat\theta^*)^\intercal
  (\phi\circ T) (\theta^*)$. Claim~4 now follows by~\eqref{eq:40new},
  the other implication being analogous.\claimend

  \goodbreak \medskip

  Hence, by Claim~4., $\theta^*$ is an equilibrium point in the sense
  of Definition~\ref{def:Eq}. To prove property~\eqref{eq:9}, by
  Claim~2 if $\widehat\theta^*_i =0$, we have:
  \begin{flalign*}
    (\bar\phi \circ \widehat{T}_i) (\theta^*) %
    & \geq %
    (\widehat \theta^*)^\intercal\, (\phi \circ {\widehat
      T})(\theta^*)
    \\
    & = %
    \sum_{j\colon \widehat\theta^*_j >0} \widehat\theta^*_j \;
    (\bar\phi\circ \widehat T_j) (\theta^*)
    \\
    & = %
    (\bar\phi \circ \widehat T_j) (\theta^*) \quad \forall\,j \colon
    \widehat\theta^*_j >0%
    & [\mbox{By Claim~4}]
  \end{flalign*}
  and the monotonicity of $\bar\phi$ ensures that~\eqref{eq:9} holds.
\end{proofof}

Observe that by the above proof any Nash equilibrium is a fixed point
for $\mathcal{F}$ in~\eqref{eq:15}.

\begin{proofof}{Proposition~\ref{prop:unique0}}
  If $\theta^*$ is any Nash equilibrium, define the sets of indexes
  \begin{equation}
    \label{eq:20}
    \begin{array}{rcl}
      \widehat I (\theta^*)
      & \colonequals
      & \left\{
        i \in \{1, \ldots, \widehat n\} \colon
        \widehat{T}_i (\theta^*) = {\widehat{T}_M} (\theta^*)
        \right\}
      \\
      \widecheck I (\theta^*)
      & \colonequals
      & \left\{
        i \in \{1, \ldots, \widecheck n\} \colon
        \widecheck{T}_i (\theta^*) = {\widecheck{T}_M} (\theta^*)
        \right\}
    \end{array}
  \end{equation}
  and note that, by Definition~\ref{def:candiate},
  \begin{eqnarray*}
    i \in \{1, \ldots, \widehat n\} \setminus \widehat I (\theta^*)
    & \implies
    & \widehat\theta^*_i = 0
      \quad \mbox{ and } \quad
      \widehat{T}_i (\theta^*) > {\widehat{T}_M} (\theta^*) \,;
    \\
    i \in \{1, \ldots, \widecheck n\} \setminus \widecheck I (\theta^*)
    & \implies
    & \widecheck \theta^*_i = 0
      \quad \mbox{ and } \quad
      \widecheck{T}_i (\theta^*) > {\widecheck{T}_M} (\theta^*).
  \end{eqnarray*}
  Recall that $\theta'$ and $\theta''$ are Nash equilibria. With
  reference to~\eqref{eq:20}, introduce the sets
  \begin{displaymath}
    \begin{array}{c}
      \mathcal{\widehat I} \colonequals \widehat I (\theta') \cap \widehat I (\theta'') \,,\qquad
      \widehat{\mathcal{J}} \colonequals \widehat I(\theta') \setminus \widehat I (\theta'') \,,\qquad
      \widehat{\mathcal{K}} \colonequals \widehat I (\theta'') \setminus \widehat I (\theta')
      \\
      \mathcal{\widecheck I} \colonequals \widecheck I (\theta') \cap \widecheck I (\theta'') \,,\qquad
      \widecheck{\mathcal{J}} \colonequals \widecheck I(\theta') \setminus \widecheck I (\theta'') \,,\qquad
      \widecheck{\mathcal{K}} \colonequals \widecheck I (\theta'') \setminus \widecheck I (\theta')
    \end{array}
  \end{displaymath}
  so that
  \begin{eqnarray}
    \label{eq:43}
    &
    & \widehat{\mathcal{I}} \cup \widehat{\mathcal{J}} \cup
      \widehat{\mathcal{K}} = \{1, \ldots,n\} \,,
    \\
    \label{eq:16}
    i
    \in
    \widehat{\mathcal{I}}
    & \implies
    & \widehat{T}_i (\theta')
      =
      {\widehat{T}_M} (\theta')
      \qquad
      \widehat{T}_i (\theta'')
      =
      {\widehat{T}_M} (\theta'') \,,
    \\
    \label{eq:18}
    j
    \in
    \widehat{\mathcal{J}}
    & \implies
    & \widehat{T}_j (\theta')
      =
      {\widehat{T}_M} (\theta')
      \qquad
      \widehat{T}_j (\theta'')
      >
      {\widehat{T}_M} (\theta'')
      \qquad
      \theta''_j
      =
      0\,,
    \\
    \label{eq:44}    k
    \in
    \widehat{\mathcal{K}}
    & \implies
    & \widehat{T}_k (\theta')
      >
      {\widehat{T}_M} (\theta')
      \qquad
      \widehat{T}_k (\theta'')
      =
      {\widehat{T}_M} (\theta'')
      \qquad
      \theta'_k
      =
      0\,.
  \end{eqnarray}
  We then have
  \begin{flalign*}
    & (\widehat\theta'' - \widehat\theta')^\intercal \,
    \left(\widehat{T} (\theta'') - \widehat{T} (\theta')\right) %
    \\
    = & \sum_{i \in \widehat{\mathcal{I}}} (\widehat\theta''_i -
    \widehat\theta'_i) \, \left(\widehat{T}_i (\theta'') -
      \widehat{T}_i (\theta')\right)%
    & [\mbox{By~\eqref{eq:43}}]
    \\
    & + \sum_{j \in \widehat{\mathcal{J}}} (\widehat\theta''_j -
    \widehat\theta'_j) \, \left(\widehat{T}_j (\theta'') -
      \widehat{T}_j (\theta')\right) + \sum_{k \in
      \widehat{\mathcal{K}}} (\widehat\theta''_k - \widehat\theta'_k)
    \, \left(\widehat{T}_k (\theta'') - \widehat{T}_k (\theta')\right)
    \\
    = & \sum_{i \in \widehat{\mathcal{I}}} (\widehat\theta''_i -
    \widehat\theta'_i) \, \left({\widehat{T}_M} (\theta'') -
      {\widehat{T}_M} (\theta')\right) %
    & [\mbox{By~\eqref{eq:16}}]
    \\
    & + \sum_{j \in \widehat{\mathcal{J}}} - \widehat\theta'_j \,
    \left(\widehat{T}_j (\theta'') - {\widehat{T}_M}
      (\theta')\right) %
    + %
    \sum_{k \in \widehat{\mathcal{K}}} \widehat\theta''_k \,
    \left({\widehat{T}_M} (\theta'') - \widehat{T}_k
      (\theta')\right) %
    & [\mbox{By~\eqref{eq:18} and~\eqref{eq:44}}]
    \\
    \leq & \sum_{i \in \widehat{\mathcal{I}}} (\widehat\theta''_i -
    \widehat\theta'_i) \, \left({\widehat{T}_M} (\theta'') -
      {\widehat{T}_M} (\theta')\right)
    \\
    & + \sum_{j \in \widehat{\mathcal{J}}} - \widehat\theta'_j \,
    \left({\widehat{T}_M} (\theta'') - {\widehat{T}_M}
      (\theta')\right) %
    + %
    \sum_{k \in \widehat{\mathcal{K}}} \widehat\theta''_k \,
    \left({\widehat{T}_M} (\theta'') - {\widehat{T}_M}
      (\theta')\right) %
    & [\mbox{By~\eqref{eq:18} and~\eqref{eq:44}}]
    \\
    = & \left( \sum_{i \in \widehat{\mathcal{I}} \cup
        \widehat{\mathcal{K}}} \widehat\theta''_i - \sum_{i \in
        \widehat{\mathcal{I}} \cup \widehat{\mathcal{J}}}
      \widehat\theta'_i \right) \left({\widehat{T}_M} (\theta'') -
      {\widehat{T}_M} (\theta')\right)
    \\
    = & \; 0 \,.
  \end{flalign*}
  An entirely similar procedure yields that also
  $(\widecheck\theta'' - \widecheck\theta')^\intercal \,
  \left(\widecheck{T} (\theta'') - \widecheck{T} (\theta')\right) =
  0$.
\end{proofof}

\section{Conclusions}
\label{sec:conclusions}

The present framework is amenable to a variety to extensions. For
instance, the presence of a two way road in the network fits in the
present setting by considering the two directions as two independent
roads with reverse orientation. The present construction applies to
the so obtained undirected graphs. However, in this model the travel
times in each direction result to be independent from the traffic in
the opposite direction. Hence, a possible extension may consist in
road travel times depending also on the traffic density along other
roads.

Inherent to the dynamics of vehicular traffic is the presence of
stochastic disturbances. Therefore, an extension of the present
setting to that of stochastic games~\cite{Shapley} is definitely worth
pursueing, possibly on the basis of the recent works~\cite{cite-key,
  10537108}.

As it is well known, the complexity of the actual computation of Nash
equilibria grows incredibly fast as the numbers of routes in the
network and of populations grow,
see~\cite[Chapter~2]{NisaRougTardVazi07}
or~\cite[Chapter~14]{PAPADIMITRIOU2015779}. Notably, Nash equilibria
can be found in entirely different ways, ranging from convex
optimization methods to those based on Brouwer fixed point
theorem. Nevertheless, the complexity of all of these methods are
essentially equivalent, see~\cite{doi:10.1137/070699652}.

\appendix

\section{Appendix: A Uniqueness Result}
\label{sec:uniqueness}
We now ensure the uniqueness of Nash equilibria. A key role is played
by Condition~\ref{it:IR} together with a sort of non-degeneracy
assumption, which we require together with assumption~\ref{item:E1}.

\begin{lemma}
  \label{lem:defPos}
  Let $\widehat Q, \widehat P, \widecheck Q$ and $\widecheck P$ be
  $N \times N$ diagonal matrices with non-negative entries. Then,
  \begin{displaymath}
    \left[
      \begin{array}{@{}cc@{}}
        \widehat Q
        & \widehat P
        \\
        \widecheck P
        & \widehat Q
      \end{array}
    \right]
    {\mbox{is positive}\atop\mbox{semidefinite}}
    {\iff}
    \forall \, h \in \{1, \ldots, N\}
    \begin{array}{l@{\!\!\!\!\!\!}c@{}}
      \mbox{either: }
      & \widehat Q_{hh} {=} 0 ,\;
        \widecheck Q_{hh} {=} 0,\;
        \widehat P_{hh} {=} 0 ,\;
        \widecheck P_{hh} {=} 0;
      \\
      \mbox{or: }
      & \widehat Q_{hh} {>} 0 ,\;
        \widecheck Q_{hh} {=} 0,\;
        \widehat P_{hh} {=} 0 ,\;
        \widecheck P_{hh} {=} 0;
      \\
      \mbox{or: }
      & \widehat Q_{hh} {=} 0 ,\;
        \widecheck Q_{hh} {>} 0,\;
        \widehat P_{hh} {=} 0 ,\;
        \widecheck P_{hh} {=} 0;
      \\
      \mbox{or: }
      & \widehat Q_{hh} {>} 0 ,\;\, \widecheck Q_{hh} {>} 0,\;
        4 \, \widehat Q_{hh} \, \widecheck Q_{hh}
        {\geq}
        (\widehat P_{hh} + \widecheck P_{hh})^2.
    \end{array}
  \end{displaymath}
\end{lemma}

Note that in the case of more than 2 populations the above statement
remains substantially unaltered, while a quite different proof is
necessary.

\begin{proofof}{Lemma~\ref{lem:defPos}}
  Write a $v \in \reali^{2N}$ as a pair
  $\left[{\widehat v \atop \widecheck v}\right]$ with
  $\widehat v , \widecheck v \in \reali^N$. Then,
  \begin{equation}
    \label{eq:5}
    \left[\widehat v^\intercal \quad \widecheck v^\intercal\right]
    \left[
      \begin{array}{cc}
        \widehat Q
        & \widehat P
        \\
        \widecheck P
        & \widehat Q
      \end{array}
    \right]
    \left[
      \begin{array}{c}
        \widehat v
        \\
        \widecheck v
      \end{array}
    \right]
    =
    \sum_{h=1}^N
    \left(
      \widehat Q_{hh} \, (\widehat v_h)^2
      +
      (\widehat P_{hh} + \widecheck P_{hh})
      \widehat v_h \, \widecheck v_h
      + \widecheck Q_{hh} \, (\widecheck v_h)^2
    \right).
  \end{equation}
  Hence, the form
  $\left[{\widehat Q \; \widehat P \atop \widecheck P \; \widehat
      Q}\right]$ is positive semidefinite if and only if each summand
  in the right hand side above is a non-negative second order
  polynomial, i.e., if and only if any one of the conditions in the
  hypothesis holds.
\end{proofof}

\begin{theorem}
  \label{thm:uniqueness}
  Let~\ref{it:IR}, \ref{item:E1} hold. Define for
  $h \in \{1, \ldots, N\}$,
  $\widehat\theta, \widehat\theta', \widehat\theta'' \in S^{\widehat
    n}$,
  $\widecheck \theta, \widecheck\theta', \widecheck\theta'' \in
  S^{\widecheck n}$,
  \begin{equation}
    \label{eq:38}
    \begin{array}{rcl}
      \widehat Q_h (\widehat\theta', \widehat\theta'', \widecheck \theta)
      & =
      & \displaystyle\int_0^1
        \partial_{\widehat\eta}\widehat\tau_h
        \left(
        \widehat\Gamma
        \left(s \, \widehat\theta'' + (1-s) \widehat\theta'\right),
        \widecheck\Gamma \, \widecheck\theta
        \right)
        \d{s}
      \\
      \widehat P_h (\widehat\theta, \widecheck\theta', \widecheck\theta'')
      & =
      & \displaystyle\int_0^1
        \partial_{\widecheck\eta}\widehat\tau_h
        \left(
        \widehat\Gamma \, \widehat\theta,
        \widecheck\Gamma
        \left((1-s)\widecheck\theta''+s\widecheck \theta'\right)
        \right)
        \d{s}
      \\
      \widecheck P_h (\widehat\theta', \widehat\theta'', \widecheck \theta)
      & =
      & \displaystyle\int_0^1
        \partial_{\widehat\eta}\widecheck\tau_h
        \left(
        \widehat\Gamma
        \left(s \, \widehat\theta'' + (1-s) \, \widehat\theta'\right),
        \widecheck\Gamma \, \widecheck\theta
        \right)
        \d{s}
      \\
      \widecheck Q_h (\widehat\theta, \widecheck\theta', \widecheck\theta'')
      & =
      & \displaystyle\int_0^1
        \partial_{\widecheck\eta}\widecheck\tau_h
        \left(
        \widehat\Gamma \, \widehat\theta,
        \widecheck\Gamma
        \left(s \, \widecheck\theta''+ (1-s) \,\widecheck \theta'\right)
        \right)
        \d{s}
    \end{array}
  \end{equation}
  Assume moreover that
  \begin{enumerate}[label={\bf(H)}]
  \item \label{item:2} For all
    $\widehat\theta', \widehat\theta'' \in S^{\widehat n}$,
    $\widecheck\theta', \widecheck\theta'' \in S^{\widecheck n}$, and
    for all $h \in \{1, \ldots, N\}$, with the exception of at most
    one $\bar h$
    \begin{enumerate}[label={\bf(H0)}]
    \item \label{item:11}
      $\widehat Q_h(\widehat\theta', \widehat\theta'', \widecheck
      \theta'') > 0$,
      $\widecheck Q_h (\widehat\theta', \widecheck\theta',
      \widecheck\theta'')>0$ and
      \\
      $4 \, \widehat Q_h(\widehat\theta', \widehat\theta'', \widecheck
      \theta'') \; \widecheck Q_h (\widehat\theta', \widecheck\theta',
      \widecheck\theta'') > \left( \widehat P_h (\widehat\theta',
        \widecheck\theta', \widecheck\theta'') + \widecheck P_h
        (\widehat\theta'', \widehat\theta', \widecheck \theta'')
      \right)^2$,
    \end{enumerate}
    while at $\bar h$ one of the following conditions holds:
    \begin{enumerate}[label={\bf(H\arabic*)}]
    \item\label{item:7}
      $\widehat Q_h(\widehat\theta', \widehat\theta'', \widecheck
      \theta'') = 0$,
      $\widecheck Q_h (\widehat\theta', \widecheck\theta',
      \widecheck\theta'') = 0$,
      $\widehat P_h (\widehat\theta', \widecheck\theta',
      \widecheck\theta'') = 0$ and
      $\widecheck P_h (\widehat\theta'', \widehat\theta', \widecheck
      \theta'') = 0$;
    \item\label{item:8}
      $\widehat Q_h(\widehat\theta', \widehat\theta'', \widecheck
      \theta'') >0$,
      $\widecheck Q_h (\widehat\theta', \widecheck\theta',
      \widecheck\theta'') = 0$,
      $\widehat P_h (\widehat\theta', \widecheck\theta',
      \widecheck\theta'') = 0$ and
      $\widecheck P_h (\widehat\theta'', \widehat\theta', \widecheck
      \theta'') = 0$;
    \item\label{item:9}
      $\widehat Q_h(\widehat\theta', \widehat\theta'', \widecheck
      \theta'') = 0$,
      $\widecheck Q_h (\widehat\theta', \widecheck\theta',
      \widecheck\theta'') > 0$,
      $\widehat P_h (\widehat\theta', \widecheck\theta',
      \widecheck\theta'') = 0$ and
      $\widecheck P_h (\widehat\theta'', \widehat\theta', \widecheck
      \theta'') = 0$;
    \item\label{item:10}
      $\widehat Q_h(\widehat\theta', \widehat\theta'', \widecheck
      \theta'') > 0$,
      $\widecheck Q_h (\widehat\theta', \widecheck\theta',
      \widecheck\theta'')>0$ and\\
      $4 \, \widehat Q_h(\widehat\theta', \widehat\theta'', \widecheck
      \theta'') \; \widecheck Q_h (\widehat\theta', \widecheck\theta',
      \widecheck\theta'') = \left( \widehat P_h (\widehat\theta',
        \widecheck\theta', \widecheck\theta'') + \widecheck P_h
        (\widehat\theta'', \widehat\theta', \widecheck \theta'')
      \right)^2$.
    \end{enumerate}
  \end{enumerate}
  Then, there exists at most one Nash equilibrium.
\end{theorem}

\begin{proofof}{Theorem~\ref{thm:uniqueness}}
  Assume $\theta'$ and $\theta''$ are Nash equilibria and using
  Lemma~\ref{lem:diff} write
  \begin{eqnarray*}
    \widehat{T} (\theta'') - \widehat{T} (\theta')
    & =
    & \widehat{T} (\widehat\theta'',\widecheck\theta'')
      - \widehat{T} (\widehat\theta',\widecheck\theta')
    \\
    & =
    & \widehat{T} (\widehat\theta'',\widecheck\theta'')
      - \widehat{T} (\widehat\theta',\widecheck\theta'')
      + \widehat{T} (\widehat\theta',\widecheck\theta'')
      - \widehat{T} (\widehat\theta',\widecheck\theta')
    \\
    & =
    & \widehat\Gamma^\intercal \, \widehat Q \,\widehat \Gamma
      (\widehat \theta'' - \widehat \theta')
      +
      \widehat\Gamma^\intercal \, \widehat P \, \widecheck \Gamma
      (\widecheck\theta'' - \widecheck\theta')
    \\
    & =
    & \left[
      \widehat\Gamma^\intercal \, \widehat Q \,\widehat \Gamma
      \quad
      \widehat\Gamma^\intercal \, \widehat P \, \widecheck \Gamma
      \right]
      \;
      (\theta'' - \theta') \,,
  \end{eqnarray*}
  where, with the notation in~\eqref{eq:38}, the diagonal $N\times N$
  matrices $\widehat Q$ and $\widehat P$ are given by
  $\widehat Q_{hh} = \widehat Q_h (\widehat\theta', \widehat\theta'',
  \widecheck\theta'')$ and
  $\widehat P_{hh} = \widehat P_h (\widehat\theta', \widecheck\theta',
  \widecheck\theta'')$, for $h \in \{1, \ldots, N\}$.  Similarly,
  $\widecheck T (\theta'') - \widecheck T (\theta') = \left[
    \widecheck\Gamma^\intercal \, \widecheck P \,\widehat \Gamma \quad
    \widecheck\Gamma^\intercal \, \widecheck Q \, \widecheck \Gamma
  \right] \; (\theta'' - \theta')$, where, with the notation
  in~\eqref{eq:38}, the diagonal $N\times N$ matrices $\widecheck Q$
  and $\widecheck P$ are given by
  $\widecheck Q_{hh} = \widecheck Q_h (\widehat\theta',
  \widecheck\theta', \widecheck\theta'')$ and
  $\widecheck P_{hh} = \widecheck P_h (\widehat\theta'',
  \widecheck\theta', \widecheck\theta'')$, for
  $h \in \{1, \ldots, N\}$.

  Hence,
  \begin{eqnarray}
    \nonumber
    \left[
    \begin{array}{c}
      \widehat{T} (\theta'') - \widehat{T} (\theta')
      \\
      \widecheck T (\theta'') - \widecheck T (\theta')
    \end{array}
    \right]
    & =
    & \left[
      \begin{array}{c}
        \widehat\Gamma^\intercal\, \widehat Q \,\widehat \Gamma
        \quad
        \widehat\Gamma^\intercal \, \widehat P \, \widecheck \Gamma
        \\
        \widecheck\Gamma^\intercal \, \widecheck P \,\widehat \Gamma
        \quad
        \widecheck\Gamma^\intercal \, \widecheck Q \, \widecheck \Gamma
      \end{array}
    \right]
    \quad
    (\theta'' - \theta')
    \\
    \label{eq:48}
    & =
    & \left[
      \begin{array}{cc}
        \widehat\Gamma^\intercal
        & 0
        \\
        0
        & \widecheck\Gamma^\intercal
      \end{array}
          \right]
          \;
          \left[
          \begin{array}{cc}
            \widehat Q
            & \widehat P
            \\
            \widecheck P
            & \widecheck Q
          \end{array}
              \right]
              \;
              \left[
              \begin{array}{cc}
                \widehat\Gamma
                & 0
                \\
                0
                & \widecheck\Gamma
              \end{array}
                  \right]
                  \;(\theta'' - \theta') \,.
  \end{eqnarray}

  \paragraph{Claim~1: The vectors
    $v = \left[{{\widehat v} \atop {\widecheck v}}\right]$ isotropic
    with respect to
    $v \mapsto v^\intercal \, \left[{{\widehat Q \; \widehat P}
        \atop{\widecheck P \; \widecheck Q}} \right] \,v$ are such
    that at most one component of $\widehat v$ is non-zero and at most
    one component of $\widecheck v$ is non-zero. }
  The form
  $v \mapsto v^\intercal \, \left[{{\widehat Q \; \widehat P}
      \atop{\widecheck P \; \widecheck Q}} \right] \,v$ is positive
  semidefinite on $\reali^{2N}$ by Lemma~\ref{lem:defPos}, which can
  be applied by~\ref{item:2}. Call $\mathcal{I}$ the subset of
  $\reali^{2N}$ consisting of isotropic vectors for the quadratic,
  though not necessarily symmetric, form
  $v \mapsto v^\intercal \, \left[{{\widehat Q \; \widehat P}
      \atop{\widecheck P \; \widecheck Q}} \right] \,v$. Call
  $v = \left[{{\widehat v} \atop {\widecheck v}}\right]$ a vector in
  $\mathcal{I}$. Fix $h \in \{1, \ldots, N\}$.  If~\ref{item:11}
  holds, then $\widehat v_h = 0$ and $\widecheck v_h = 0$.  Hence, if
  there is no $\bar h$ where equality in~\ref{item:2} holds, then
  necessarily $\mathcal{I} = \{0\}$ and the Claim follows.

  If there is a $\bar h$, where~\ref{item:11} does not hold, then $4$
  cases are in order, and by means of the representation~\eqref{eq:5},
  we have:
  \begin{description}
  \item[Case~\ref{item:7}:]
    $\mathcal{I} = {\rm span} \{e_{\bar h}, e_{N+\bar h}\}$.
  \item[Case~\ref{item:8}:]
    $\mathcal{I} = {\rm span} \{e_{N+\bar h}\}$.
  \item[Case~\ref{item:9}:] $\mathcal{I} = {\rm span} \{e_{\bar h}\}$.
  \item[Case~\ref{item:10}:]
    $\mathcal{I} = {\rm span} \left\{\sqrt{\widecheck Q_{\bar h}} \;
      e_{\bar h} - \sqrt{\widehat Q_{\bar h}} \; e_{N+\bar h}
    \right\}$.
  \end{description}
  Therefore, any vector
  $v = \left[{{\widehat v} \atop {\widecheck v}}\right]$ in
  $\mathcal{I}$ is such that at most one component of $\widehat v$ is
  non-zero and at most one component of $\widecheck v$ is
  non-zero.\claimend

  \paragraph{Claim~2: If
    $\theta'\equiv(\widehat\theta',\widecheck\theta'), \,
    \theta''\equiv(\widehat\theta'', \widehat\theta'') \in S^N$ and
    $\theta' \neq \theta''$, then at least one of the $2$ vectors
    $\widehat\Gamma(\widehat\theta''-\widehat\theta')$ or
    $\widecheck\Gamma(\widecheck\theta''-\widecheck\theta')$ has at
    least $2$ non-zero components.}

  We are assuming $\theta' \neq \theta''$. Then,
  $\widehat\theta' \neq \widehat\theta''$ or
  $\widecheck\theta' \neq \widecheck\theta''$.  Since
  $\sum_{i=1}^{\widehat n} \widehat\theta'_i = 1 =
  \sum_{j=1}^{\widecheck n} \widehat\theta''_j$, at least one of the
  $2$ vectors $\widehat\theta'' - \widehat\theta'$ or
  $\widecheck\theta'' - \widecheck\theta'$ has at least $2$ non-zero
  components.

  By Condition~\ref{it:IR} and Remark~\ref{rem:perm}, $\widehat\Gamma$
  admits the decomposition $\widehat\Gamma = \left[
    \begin{array}{c}
      \Id_{\widehat n}
      \\ \hdashline[2pt/3pt]
      \widehat\gamma
    \end{array}
  \right]$, where the matrix $\widehat\gamma$ has order
  $(N-\widehat n) \times \widehat n$ and its entries are either $0$ or
  $1$. Similarly, $\widecheck\Gamma$ admits the decomposition
  $\widecheck\Gamma = \left[
    \begin{array}{c}
      \Id_{\widecheck n}
      \\ \hdashline[2pt/3pt]
      \widecheck\gamma
    \end{array}
  \right]$, where the matrix $\widecheck\gamma$ has order
  $(N-\widecheck n) \times \widecheck n$ and its entries are either
  $0$ or $1$.

  Thus,
  \begin{displaymath}
    \widehat\Gamma (\widehat\theta'' - \widehat\theta')
    =
    \left[
      \begin{array}{c}
        \widehat\theta''-\widehat\theta'
        \\ \hdashline[2pt/3pt]
        \widehat\gamma (\widehat\theta''-\widehat\theta')
      \end{array}
    \right]
    \quad \mbox{ and similarly} \quad
    \widecheck\Gamma (\widecheck\theta'' - \widecheck\theta')
    =
    \left[
      \begin{array}{c}
        \widecheck\theta''-\widecheck\theta'
        \\ \hdashline[2pt/3pt]
        \widecheck\gamma (\widecheck\theta''-\widecheck\theta')
      \end{array}
    \right].
  \end{displaymath}
  This proves that at least one of the $2$ vectors
  $\widehat\Gamma(\widehat\theta''-\widehat\theta')$ or
  $\widecheck\Gamma(\widecheck\theta''-\widecheck\theta')$ has at
  least $2$ non-zero components.\claimend

  \paragraph{Claim~3: If $\theta',\theta''$ are distinct Nash equilibria, then
    \begin{equation}
    \label{eq:45}
    \begin{array}{c}
      (\widehat \theta'' - \widehat \theta')^\intercal
      \;
      \left(\widehat{T} (\theta'') - \widehat{T} (\theta')\right)
      \geq
      0
      \,,\qquad
      (\widecheck \theta'' - \widecheck \theta')^\intercal
      \;
      \left(\widecheck{T} (\theta'') - \widecheck{T} (\theta')\right)
      \geq
      0
      \\
      \mbox{and }\quad
      \max\left\{
      (\widehat \theta'' - \widehat \theta')^\intercal
      \;
      \left(\widehat{T} (\theta'') - \widehat{T} (\theta')\right)
      \,,\;
      (\widecheck \theta'' - \widecheck \theta')^\intercal
      \;
      \left(\widecheck{T} (\theta'') - \widecheck{T} (\theta')\right)
      \right\}
      > 0 \,.
    \end{array}
  \end{equation}}
The form
$v \mapsto v^\intercal \left[ {{\widehat\Gamma^\intercal \;\; 0}\atop
    {0 \;\; \widecheck\Gamma^\intercal}}\right] \, \left[{{\widehat Q
      \; \widehat P} \atop{\widecheck P \; \widecheck Q}} \right] \,
\left[ {{\widehat\Gamma \; 0}\atop {0 \; \widecheck\Gamma}}\right]v$
is positive semidefinite on $\reali^{\widehat n + \widecheck
  n}$. Indeed, if there exists a
$v \in \reali^{\widehat n + \widecheck n}$ such that
$v^\intercal \left[ {{\widehat\Gamma^\intercal \;\; 0}\atop {0 \;\;
      \widecheck\Gamma^\intercal}}\right] \, \left[{{\widehat Q \;
      \widehat P} \atop{\widecheck P \; \widecheck Q}} \right] \,
\left[ {{\widehat\Gamma \; 0}\atop {0 \;
      \widecheck\Gamma}}\right]v<0$, then
$\left(\left[ {{\widehat\Gamma \; 0}\atop {0 \;\;
        \widecheck\Gamma}}\right] v\right)^\intercal \,
\left[{{\widehat Q \; \widehat P} \atop{\widecheck P \; \widecheck Q}}
\right] \, \left(\left[ {{\widehat\Gamma \; 0}\atop {0 \;
        \widecheck\Gamma}}\right]v\right)<0$, which contradicts the
fact that
$v \mapsto v^\intercal \, \left[{{\widehat Q \; \widehat P}
    \atop{\widecheck P \; \widecheck Q}} \right] \,v$ is positive
semidefinite on $\reali^{2N}$.

Call $\mathcal{J}$ the subset of $\reali^{\widehat n + \widecheck n}$
consisting of the vectors $v$ isotropic relative to the form
$v \mapsto v^\intercal \left[ {{\widehat\Gamma^\intercal \;\; 0}\atop
    {0 \;\; \widecheck\Gamma^\intercal}}\right] \, \left[{{\widehat Q
      \; \widehat P} \atop{\widecheck P \; \widecheck Q}} \right] \,
\left[ {{\widehat\Gamma \; 0}\atop {0 \;
      \widecheck\Gamma}}\right]v$. Clearly, $v \in \mathcal{J}$ with
$v = \left[{{\widehat v} \atop {\widecheck v}}\right]$, if and only if
$\left[{{\widehat\Gamma \, \widehat v}\atop{\widecheck\Gamma \,
      \widecheck v}} \right] \in \mathcal{I}$.

This shows that
$\left[{{\widehat\Gamma \, \widehat v}\atop{\widecheck\Gamma \,
      \widecheck v}} \right]$ can not be in $\mathcal{I}$. Therefore
\begin{displaymath}
  (\theta'' - \theta')^\intercal \;
  \left[ {{\widehat\Gamma^\intercal \;\;
        0}\atop {0 \;\; \widecheck\Gamma^\intercal}}\right] \;
  \left[{{\widehat Q \;\; \widehat P} \atop{\widecheck P \;\; \widecheck
        Q}} \right] \; \left[ {{\widehat\Gamma \;\; 0}\atop {0 \;\;
        \widecheck\Gamma}}\right] \; (\theta'' - \theta')
  > 0
\end{displaymath}
and by~\eqref{eq:48} we have~\eqref{eq:45}.

The contradiction between Claim~3 and Proposition~\ref{prop:unique0}
shows that $\widehat\theta' = \widehat\theta''$.
\end{proofof}

\paragraph*{Acknowledgments:}
RMC and FM were partially supported by the INdAM-GNAMPA~2023 project
\emph{Analytic Techniques for Biological Models, Fluid Dynamics and
  Vehicular Traffic} and acknowledge the PRIN~2022 project
\emph{Modeling, Control and Games through Partial Differential
  Equations} (D53D23005620006), funded by the European Union - Next
Generation EU. LG was partially supported by the INdAM-GNSAGA~2022
project \emph{Algebraic Structures, Geometry and Applications}.

{ \small

  \bibliography{ColomboGiuzziMarcellini_revised}

  \bibliographystyle{abbrv}

}

\paragraph{Data Availability Statement:}
This manuscript has no associated data.

\end{document}